\documentclass[]{amsart}
\usepackage[margin = 3.5 cm,marginpar= 2cm]{geometry}

\usepackage{amsmath,amssymb,mathtools,subcaption,enumitem,mleftright}

\renewcommand{\left}{\mleft}
\renewcommand{\right}{\mright}
\usepackage{hyperref}

\newcommand*{\N}{\mathbb{N}}
\newcommand*{\R}{\mathbb{R}}
\newcommand*{\Dcal}{\mathcal{D}}
\newcommand*{\Mcal}{\mathcal{M}}
\newcommand*{\Ocal}{\mathcal{O}}

\newcommand*{\abs}[1]{\lvert#1\rvert}
\DeclareMathOperator*{\artanh}{artanh}
\newcommand*{\Id}{\mathrm{Id}}
\DeclareMathOperator*{\sgn}{sgn}
\DeclareMathOperator*{\meas}{meas}
\newcommand{\dee}{\mathop{}\!\mathrm{d}}
\newcommand*{\e}{\mathrm{e}}

\newcommand*{\phiv}{\varphi}

\newtheorem{thm}{Theorem}[section]

\newtheorem{prp}[thm]{Proposition}

\theoremstyle{remark}
\newtheorem{rem}{Remark}

\theoremstyle{definition}
\newtheorem{dfn}{Definition}

\title[Stumpons are non-conservative traveling waves]{Stumpons are non-conservative traveling waves of the Camassa--Holm equation} 

\author[S.\ T.\ Galtung]{Sondre Tesdal Galtung}
\address[S.\ T.\ Galtung]{Department of Mathematical Sciences, NTNU -- Norwegian University of Science and Technology, 7491 Trondheim, Norway}
\email{sondre.galtung@ntnu.no}

\author[K.\ Grunert]{Katrin Grunert}
\address[K.\ Grunert]{Department of Mathematical Sciences, NTNU -- Norwegian University of Science and Technology, 7491 Trondheim, Norway}
\email{katrin.grunert@ntnu.no}

\thanks{This work was supported by the Research Council of Norway with grant numbers 250070 {\it Waves and Nonlinear Phenomena (WaNP)} and 286822  {\it Wave Phenomena and Stability --- a Shocking Combination (WaPheS)}.}
\keywords{Camassa--Holm equation, traveling wave solutions, energy-preserving numerical methods}
\subjclass{35C07 \and 35Q51 \and 65M22}

\begin{document}
	
\maketitle

\begin{abstract}
  It is well-known that by requiring solutions of the Camassa--Holm equation to satisfy a particular local conservation law for the energy in the weak sense,
  one obtains what is known as conservative solutions.
  As conservative solutions preserve energy, one might be inclined to think that any solitary traveling wave is conservative.
  However, in this paper we prove that this is not true for the traveling waves known as stumpons.
  We illustrate this result by comparing the stumpon to simulations produced by a recently developed numerical scheme for conservative solutions.
  
\end{abstract}

  \section{Introduction}
The Camassa--Holm equation
\begin{equation}\label{eq:CH}
  u_t - u_{txx} + 2\kappa u_x + 3uu_x - 2u_xu_{xx} - uu_{xxx} = 0,
\end{equation}
where $\kappa \in \R$, has been thoroughly studied since it was brought to attention as a model for shallow-water waves in \cite{CamHol1993}.
It is first known to have appeared as a rather anonymous special case in a family of completely integrable evolution equations in \cite{FucFok1981}.
See also \cite{Johnson2002} and \cite{ConLan2009} for a thorough discussion of how the Camassa--Holm equation
can be derived from the equations of hydrodynamics.

If $u = u(t,x)$ is a solution of \eqref{eq:CH}, then
\begin{equation}\label{eq:CHshift}
  v(t,x) = u(t,x-\alpha t) + \alpha
\end{equation}
is a solution of \eqref{eq:CH} with $\kappa$ replaced by $\kappa-\alpha$. In particular, choosing $\alpha=\kappa$ yields that $v(t,x)$ is a solution of the limiting case $\kappa=0$, i.e., 
\begin{equation}\label{eq:CH0}
  u_t - u_{txx} + 3uu_x - 2u_xu_{xx} - uu_{xxx} = 0.
\end{equation}
Thus we are only going to study weak solutions for \eqref{eq:CH0}, which have bounded traveling wave profiles as initial data. A presentation of all possible weak traveling wave solutions can be found in \cite{Lenells2005}.

The extensive study of the Camassa--Holm equation can be explained by its many interesting mathematical properties.
As mentioned earlier, it is completely integrable, and so it has infinitely many conserved quantities, see, e.g., \cite{Lenells2005cons}.
These invariants take the form of functionals of the type
\begin{equation}\label{eq:invariant}
  \int_\R F(u(t,x),\partial_x u(t,x), \dots, \partial^k_x u(t,x)) \dee x,
\end{equation}
where $F$ is a polynomial in its arguments, and $\partial^k_x$ denotes partial derivatives with respect to $x$ of order $k \in \N$.
Given $m\in \R$, the perhaps most interesting of these conserved quantities is
\begin{equation}\label{eq:energy1}
  \int_\R \left((u-m)^2 +u_x^2 \right)(t,x) \dee x,
\end{equation}
which is well defined for $(u-m)(t,\cdot) \in H^1(\R)$. It is often called the energy and serves as the foundation for interpreting \eqref{eq:CH} as a geodesic equation, see \cite{Kouranbaeva1999,ConKol2003}.
Note that in the literature $m$ is often used to denote the momentum variable, i.e.,  $u-u_{xx}$,
but in this paper we follow \cite{Lenells2005}, where $m$ always denotes a constant appearing in various parameters.

Yet another interesting property, which distinguishes \eqref{eq:CH0} from certain other well-known integrable equations such as the Korteweg--de Vries
equation, is that even smooth initial data can lead to singularity formation, also known as wave breaking, in finite time, cf.\ \cite{ConEsc1998}.
Here the singularity formation corresponds to the slope of $u_x$ becoming unbounded, while $u$ remains bounded,
and so we say that the wave breaks.
This introduces an ambiguity in how to extend the solutions past wave breaking,
which led to the concepts of conservative \cite{BreCon2007cons,HolRay2007cons, GruHolRay2012} and dissipative \cite{BreCon2007diss,HolRay2009diss, GruHolRay2014} solutions. The key element in these works is to rewrite \eqref{eq:CH} in its nonlocal form, i.e.,
\begin{subequations}\label{eq:CHP}
  \begin{align}
    u_t + uu_x + P_x &= 0, \label{eq:CHPa} \\
    P-P_{xx} &= u^2 + \frac12 u_x^2, \label{eq:CHPb}
  \end{align}
\end{subequations}
where we find, using the fundamental solution of the Helmholtz operator, that
\begin{equation}\label{eq:P}
  P(t,x) = \int_\R \frac12\e^{-\abs{x-z}} \left( u^2 + \frac12 u_x^2  \right)\!(t,z)\dee{z}.
\end{equation}
Observe that the exponential kernel in \eqref{eq:P} is associated with decaying solutions of the Helmholtz equation,
but this expression is also well-defined for periodic solutions $u \in H^1_\text{per}(\R)$ due to the rapid decay of the kernel.
Indeed, for periodic $u$ one can rewrite \eqref{eq:P} as an integral over one period with a hyperbolic cosine function
as integration kernel, see Section \ref{s:cons:per}, and this is exactly the periodic framework used in, e.g.,
\cite{ConMcK1999,HolRay2008}.

Using the Lagrangian reformulation, or the method of characteristics, the equation can be further rewritten as a system of ordinary differential equations. Combining this system with yet another equation describing the time evolution of the energy along characteristics yields the different solution concepts. Each of them gives a weak solution to \eqref{eq:CHP}, but the energy is manipulated in different ways when wave breaking occurs and may not be conserved. 
The foundation for the conservative weak solution concept is the following conservation law identity for the energy,
\begin{equation}\label{eq:claw}
  (u^2+u_x^2)_t + (u(u^2+u_x^2))_x = (u^3-2Pu)_x,
\end{equation}
which can be verified to hold pointwise when $u$ is smooth.
For conservative solutions, $u^2 + u_x^2$ represents the absolutely continuous part of a measure,
and we require that this measure satisfies a conservation law which is analogous to \eqref{eq:claw}.

The concept of a conservative solution is essential for our main result,
namely that there exist traveling wave solutions of the Camassa--Holm equation, known as stumpons, which are not conservative, i.e., they do not satisfy \eqref{eq:claw} in the weak sense. This is somewhat surprising, as traveling waves are translations of an initial profile and so they must leave functionals of the form \eqref{eq:invariant} invariant. In particular, the energy \eqref{eq:energy1} is clearly preserved by such solutions.
Hence, despite traveling waves being relatively simple to express in an Eulerian framework, the techniques presented in \cite{BreCon2007cons,HolRay2007cons,GruHolRay2012} cannot be used to study various weak traveling wave solutions. 
In turn, this means that numerical methods which rely on the conservative Lagrangian formulation introduced in the above works cannot reproduce a
stumpon solution.
As a matter of fact, the inability to reproduce the stumpon with the conservative numerical method from \cite{GalGru2021},
based on the discretization in \cite{GalRay2021}, is what led to the
discovery of these results.

We emphasize the important distinction between the notion of a conservative solution used here and in \cite{BreCheZha2015}, and solutions which preserve the energy \eqref{eq:energy1}, which are often also called conservative, especially in the setting of numerical schemes.
Our main results show that the former is a proper subset of the latter.

The rest of the paper is organized as follows:
Section \ref{s:travel} gives an overview of the properties of traveling waves for the Camassa--Holm equation
and presents the most prominent examples of such solutions, i.e., peakons, cuspons, and stumpons.
Section \ref{s:cons:per} defines what it means to be a weak conservative solution and presents the main result: The periodic stumpon, in contrast to the periodic cuspon, is not a weak conservative solution. Furthermore, we investigate why the numerical results obtained by applying the method from \cite{GalGru2021} are correct for small times. Section \ref{s:cons:dec} presents how the periodic results can be extended to the non-periodic case with small modifications.

\section{Traveling waves of the Camassa--Holm equation} \label{s:travel}
Lenells \cite{Lenells2005} classified all bounded traveling waves of the Camassa--Holm equation, and so we will use several of his results in our analysis.
In this section, we will first present some properties which must be satisfied by the traveling waves, and then we give three examples of such
solutions: peakons, cuspons, and stumpons.
Here, the well-known peakon solution is included because of its simple, explicit form which allows for explicit computations.
On the other hand, the cuspon and stumpon solutions play the leading roles in our results.
We emphasize that we here only consider bounded traveling waves.

\subsection{Properties of traveling waves}
A traveling wave solution of \eqref{eq:CH0} is a solution of the form
\begin{equation}\label{eq:tw}
  u(t,x) = \phiv(x-st),
\end{equation}
where the local Sobolev function
$\phiv \in H^1_\text{loc}(\R)$ satisfies the equation
\begin{equation}\label{eq:tw_cond}
  (\phiv')^2 + 3\phiv^2  -2s\phiv = \left( (\phiv-s)^2 \right)'' + a
\end{equation}
in the distributional sense for some constant $a \in \R$, as shown in \cite{Lenells2005}.
Lemma 4 therein states that such a $\phiv$  is a traveling wave of \eqref{eq:CH0} with velocity $s$ if and only if the following properties hold:
\begin{enumerate}[label=(TW\arabic*)]
  \item \label{enum:tw1} There exist disjoint open intervals $E_i$ and a closed set $C$ such that
  $\R\!\setminus\! C = \bigsqcup_{i=1}^\infty E_i$, $\phiv \in C^\infty(E_i)$, $\phiv(x) \neq c$ for $x \in E_i$, while $\phiv(x) = s$ for $x \in C$.
  \item \label{enum:tw2} There is $a, b_i \in \R$ such that $(\phiv')^2 = F(\phiv)$ for $x \in E_i$, $\phiv \to s$ at any finite endpoint of $E_i$, where
  \begin{equation*}
    F(\phiv) = \frac{\phiv^2 (s-\phiv) + a\phiv + b_i}{s-\phiv}
  \end{equation*}
  Furthermore, if the Lebesgue measure of $C$ is strictly positive, i.e., $\meas(C) > 0$, then $a = s^2 $. 
  \item $(s-\phiv)^2 \in W^{2,1}_\text{loc}(\R)$. \label{enum:tw3}
\end{enumerate}
Note that the identity $(\phiv')^2 = F(\phiv)$ in \ref{enum:tw2} can be restated as
\begin{equation}\label{eq:dphi2}
  (\phiv')^2 = \frac{(M-\phiv)(\phiv-m)(\phiv-z)}{s-\phiv},
\end{equation}
where we have factorized
\begin{equation}\label{eq:factor}
  \phiv^2 (s-\phiv) + a\phiv + b_i = (M-\phiv)(\phiv-m)(\phiv-z)
\end{equation}
and $z = s-M-m$.
The three parameters $M$, $s$, and $m$ are exactly those used in \cite[Theorem 1]{Lenells2005} to classify the traveling waves.
As a consequence, one must have
\begin{equation}\label{eq:a:tw}
  a = -Mm -(M+m)(s -M -m).
\end{equation}
For any such traveling wave solution, one readily obtains that the associated $P(t,x)$ in \eqref{eq:P} satisfies $P(t,x) = P(0,x-st)$.
In fact, from \eqref{eq:tw_cond} and \ref{enum:tw1}--\ref{enum:tw3} we can say even more --- but first,
a formal calculation to motivate the result.
Combining \eqref{eq:CHPa}, \eqref{eq:tw}, and the fact that $P_x(t,x) = P_x(0,x-st)$ we obtain
\begin{equation*}
  -s \phiv'(x-st) + \phiv \phiv'(x-st) + P_x(0,x-st) = 0
\end{equation*}
or equivalently
\begin{equation*}
  P_x(0,x) = (s-\phiv(x))\phiv'(x) = -\frac12\left( (s-\phiv(x))^2 \right)'.
\end{equation*}
Combining this expression with \eqref{eq:CHPb} and \eqref{eq:tw_cond} (note that this is only a distributional identity!) it then follows
\begin{align*}
  P(0,x) &= \phiv^2(x) + \frac12 (\phiv'(x))^2  - \frac12 \left( (s-\phiv(x))^2 \right)'' \\
  &= \phiv^2(x) + \frac12 \left[ a - 3\phiv^2(x) + 2s \phiv(x) \right] \\
  &= \frac12 \left[ a + s^2 - \left(s-\phiv(x)\right)^2 \right].
\end{align*}
This formal calculation motivates the following result.
\begin{prp} \label{prp:P}
  For a traveling wave solution \eqref{eq:tw} of the Camassa--Holm equation, satisfying \eqref{eq:tw_cond}, the associated $P(t,x)$ defined in \eqref{eq:P} is given by
  \begin{equation}\label{eq:Ptw}
    P(t,x) = \frac{a + s^2-(s-\phiv(x-st))^2}{2}.
  \end{equation}
\end{prp}

\begin{proof}
  From the property \ref{enum:tw1} we deduce that $\phiv'(x)$ might not exist on $\partial C$, the boundary of $C$. Therefore, we split $C$ in its interior and boundary, $C = \mathring{C} \sqcup \partial C$, such that if $\mathring{C} \neq \emptyset$ we have $\phiv'(x) = 0$ for $x \in \mathring{C}$.
  Since we are in one dimension, $\partial C$ can at most be a null set, hence $\meas(\partial C) = 0$.
  Then, since $\phiv \in H^1_\text{loc}(\R)$, it follows that the contribution from $\partial C$ to the integral \eqref{eq:P} which defines $P(0,x)$ is zero.
  Consequently,
  \begin{align}\label{eq:Pcalc}
    \begin{aligned}
      P(0,x) &= \int_\R \frac12\e^{-\abs{x-z}} \left( \phiv^2 + \frac12 (\phiv')^2  \right)(z)\dee{z} \\
      &= \frac{a+s^2}{2} + \frac12 \int_\R \frac12\e^{-\abs{x-z}} \left( 2 \phiv^2 + (\phiv')^2  - a-s^2 \right)(z)\dee{z} \\
      &= \frac{a+s^2}{2} + \frac{s^2-a}{2} \int_{\mathring{C}}\frac12 \e^{-\abs{x-z}}\dee{z} \\
      &\quad- \frac12 \sum_{i=1}^{\infty} \int_{E_i} \frac12\e^{-\abs{x-z}} \left(a + s^2 - 2 \phiv^2 - (\phiv')^2  \right)(z)\dee{z},
    \end{aligned}
  \end{align}
  where we have used $\phiv(x) =s$ and $\phiv'(x) = 0$ for $x \in \mathring{C}$.
  For the final expression above we consider two cases:
  If $\meas(C) = 0$, then the second term vanishes. Otherwise, if $\meas(C) > 0$ then we have from \ref{enum:tw2} that $a = s^2 $,
  and so in any case the second term vanishes.
  
  Now we treat each term in the sum, where as $\phiv \in C^\infty(E_i)$ it is perfectly fine to use \eqref{eq:tw_cond} as a pointwise identity for $x \in E_i$, yielding
  \begin{align*}
    I_i(x) &\coloneqq \int_{E_i} \frac12\e^{-\abs{x-z}} \left( a + s^2 - 2 \phiv^2(z) - (\phiv')^2(z)  \right)\dee{z} \\
    &= \int_{E_i} \frac12\e^{-\abs{x-z}} \left( (s-\phiv(z))^2 - \left( (s-\phiv(z))^2 \right)'' \right)\dee{z}.
  \end{align*}
  Here we observe that both the Helmholtz operator $\Id - \partial^2_{x}$ and its fundamental solution $\frac12\e^{-\abs{x-z}}$ appear in the integral.
  We shall therefore integrate by parts and we write $E_i = (e^-_i, e^+_i)$ to simplify the notation below.
  There are two possible cases:
  
  \textit{(i) $x\notin E_i$}:
  We integrate the second term by parts twice to obtain
  \begin{align*}
    I_i(x) &= \left[ \e^{-\abs{x-z}} (s-\phiv(z)) \phiv'(z) \right]_{e^-_i}^{e^+_i} + \left[ \sgn(x-z) \frac12 \e^{-\abs{x-z}} (s-\phiv(z))^2 \right]_{e^-_i}^{e^+_i},
  \end{align*}
  that is, only the boundary terms remain.
  However, we find that these also vanish: From \ref{enum:tw1} and \ref{enum:tw2} it is clear that $\phiv(x)$,
  and thus also $(s-\phiv(x))F(\phiv(x))$, is bounded for $x\in E_i$.
  Now, if the endpoint $e^\pm_i$ is finite, then
  \begin{equation*}
    \abs{(s-\phiv(z))\phiv'(z)} \bigg|_{z=e^\pm_i} = \sqrt{(s-\phiv(z))^2 F(\phiv(z))} \bigg|_{z=e^\pm_i} = 0 = (s-\phiv(z))^2 \bigg|_{z=e^\pm_i}
  \end{equation*}
  due to $\phiv(z) \to s$ as $z \to e^\pm_i$. Otherwise, if the endpoint $e^\pm_i$ is not finite, then the term vanishes because $\e^{-\abs{x-z}} \to 0$ as $z \to e^\pm_i$.
  Hence, $I_i(x) = 0$.
  
  \textit{(ii) $x\in E_i$}:
  In this case we also integrate by parts, but now there is a discontinuity at $x$ when integrating for the second time, leading to the expression
  \begin{align*}
    I_i(x) &= \left[ \e^{-\abs{x-z}} (s-\phiv(z)) \phiv'(z) \right]_{e^-_i}^{e^+_i} + \left[ \frac12 \e^{-\abs{x-z}} (s-\phiv(z))^2 \right]_{e^-_i}^{x}  \\
    &\quad- \left[ \frac12 \e^{-\abs{x-z}} (s-\phiv(z))^2 \right]_{x}^{e^+_i}.
  \end{align*}
  Now, the same argument as in the previous case shows that the contribution from the boundary points $e^\pm_i$ vanishes, while we retain a contribution from $x$, namely $I_i(x) = (s-\phiv(x))^2$.
  
  In conclusion, we have found that $I_i(x) = (s-\phiv(x))^2 \chi_{E_i}(x)$, where $\chi_{E_i}(x)$ is the indicator function of the set $E_i$.
  Using that $\phiv(x) = s$ for $x \in C$ we can add zero in the form of
  $(s-\phiv(x))^2 \chi_C(x)$ to obtain
  \begin{equation*}
    \sum_{i=1}^{\infty} I_i(x) = (s-\phiv(x))^2 \sum_{i=1}^{\infty} \chi_{E_i}(x) + (s-\phiv(x))^2 \chi_C(x) = (s-\phiv(x))^2.
  \end{equation*}
  Finally, inserting this into \eqref{eq:Pcalc}, we have shown
  \begin{equation*}
    P(0,x) = \frac{a+s^2-(s-\phiv(x))^2}{2},
  \end{equation*}
  and the result follows from $P(t,x) = P(0,x-st)$.
\end{proof}

\begin{rem}
  Note that with $P(0,x)$ given by \eqref{eq:Ptw}, the identity \eqref{eq:tw_cond} is exactly \eqref{eq:CHPb} for traveling waves.
\end{rem}

From the aforementioned equivalence between solutions of \eqref{eq:CH} and \eqref{eq:CH0}, we know that if $\phiv$ is a traveling wave profile of \eqref{eq:CH0}, then
\begin{equation}\label{eq:tw0}
  u(t,x) = \phiv(x-(s-\kappa)t) - \kappa
\end{equation}
is a traveling wave solution of \eqref{eq:CH}.
Furthermore, given $\kappa$, one can check that if $\tilde{u}(t,x) = \phiv(x-st)$ is a traveling wave solution of \eqref{eq:CH0} with corresponding $\tilde{P}(t,x)$ given by \eqref{eq:Ptw}, then $P(t,x)$ corresponding to the traveling wave solution
$u(t,x) = \phiv(x-(s-\kappa)t)-\kappa$ of \eqref{eq:CH} satisfies $P(t,x) = \tilde{P}(t,x+\kappa t) - \kappa^2$.

\subsubsection{Characteristics}
The characteristics $y(t,\xi)$ associated with a solution $u(t,x)$ of \eqref{eq:CH} are given by
\begin{equation}\label{eq:char}
  y_t(t,\xi) = u(t,y(t,\xi)), \qquad y(0,\xi) = y_0(\xi).
\end{equation}
These are labeled using the parameter $\xi$ and the initial parametrization $y_0(\xi)$, and can be regarded as the position of a fluid particle labeled
$\xi$ at time $t$.
For the upcoming examples of solitary traveling waves it is illuminating to study their characteristics, as these behave in qualitatively different ways.
Moreover, the characteristics are part of the Lagrangian coordinate system, which will be the foundation for our analysis of the conservative solutions.

\subsection{Peakons}
The name \textit{peakon} comes from the \textit{peaked} crest or trough of these \textit{soliton} solutions of the Camassa--Holm equation.
One of the appealing features of the peakons are their explicit formulas, which are a nice exception among the general traveling waves,
and which allow for exact computations of associated quantities.
An important attribute of the peakons is orbital stability, both in the decaying \cite{ConStr2000} and periodic \cite{Lenells2004} cases; that is, their shape is stable under small perturbations, which allows for the detection of these
distinctive wave profiles.
Another interesting property of peakons is that they can be combined to obtain multi-peakon solutions,
for which the dynamics can be described by a system of ordinary differential equations: see \cite{Camassa2003} for a Hamiltonian system formulation, and \cite{HolRay2007MP} for a global-in-time Lagrangian formulation.
In fact, energy-preserving multi-peakon solutions are examples of completely integrable systems, cf.\ \cite{EckKos2014,EckKos2018}.

Below we will present the peakon solutions both on the full line and in the periodic case.

\subsubsection{Peakon with decay}
According to the classification in \cite{Lenells2005}, the peakons with decay for \eqref{eq:CH0} correspond to $z =-m=m=0$ and $s = M$ with either 
\begin{equation}\tag{d} 0 < s, \quad \max\limits_{x\in\R} \phiv(x) = s, \quad  \text{and} \quad \lim\limits_{x\to\pm\infty}\phiv(x) \searrow 0 \enspace  \text{exponentially},
\end{equation}
or 
\begin{equation*}\tag{d'}
  s<0, \quad \min\limits_{x\in\R} \phiv(x) = s, \quad \text{and} \quad \lim\limits_{x\to\pm\infty}\phiv(x) \nearrow 0 \enspace \text{exponentially}.
\end{equation*}
These parameter values reduce \eqref{eq:dphi2} to $(\phiv')^2 = \phiv^2$.
Moreover, the explicit expression
\begin{equation}\label{eq:peakon_dec}
  u(t,x) =  s \e^{-\abs{x-x_0-st}},
\end{equation}
is given in \cite[Equation (8.9)]{Lenells2005}, see also \cite{CamHol1993}.

Let us for simplicity assume that we are in the case where $s > 0$ such that the peak located at $x=x_0$ at time $t=0$ is a crest for the wave.
From \eqref{eq:peakon_dec} we can compute the exact characteristics.
Fixing $\xi$ and defining $w(t) = y(t,\xi)-st-x_0$,  \eqref{eq:char} reads \begin{equation*}
  w'(t) = s \left( \e^{-\abs{w(t)}} - 1 \right), \qquad z(0) = y_0(\xi) - x_0,
\end{equation*}
which has a Lipschitz-continuous right-hand side and hence its solutions are unique.
Noting that $w = 0$ is a constant solution, and that $w'(t) < 0$ whenever $w(t) \neq 0$, it is sufficient to consider the three cases $w(0) > 0$, $w(0) = 0$, and $w(0) < 0$.
In particular, we obtain 
\begin{equation*}
  w(t) = \sgn(w(0)) \ln\left(1+\left(\e^{\abs{w(0)}}-1 \right)\e^{-\sgn(w(0))st} \right),
\end{equation*}
where we use the convention $\sgn{0} = 0$.
Note that the peak follows a single characteristic, which acts as an asymptote for the characteristics located to the left and the right:
\begin{equation*}
  \lim\limits_{t\to\infty} y(t,\xi) - (x_0 + st) = \begin{cases}
    0, & y_0(\xi) > x_0 \\ -\infty, & y_0(\xi) < x_0.
  \end{cases}
\end{equation*}

\begin{rem}
  As a check of the formula \eqref{eq:Ptw}, we can consider a peakon with decay for \eqref{eq:CH0}, for which $\phiv(x) = s \e^{-\abs{x-x_0}}$. Then we can explicitly compute $P(0,x)$ from \eqref{eq:P},
  and we obtain exactly $P(0,x) = s^2(1-(1-\e^{-\abs{x-x_0}})^2)/2$ as given by \eqref{eq:Ptw}.
\end{rem}

\subsubsection{Periodic peakon}
In analogy with the decaying peakons, the periodic peakons correspond to two different cases in the classification in \cite{Lenells2005}.
Here $s = M$ and $z=-m$ with either
\begin{equation}\tag{c}
  0< \min\limits_{x\in\R} \phiv(x) = m < s = \max\limits_{x\in\R}\phiv(x)
\end{equation} 
or 
\begin{equation}\tag{c'}  0 >\max\limits_{x\in\R} \phiv(x) = m > s = \min\limits_{x\in\R}\phiv(x).
\end{equation}
Then \eqref{eq:dphi2} becomes $(\phiv')^2 = (\phiv-m)(\phiv-z)$. Moreover, the solution, see \cite{Lenells2005}, is periodic with period $2L$, where 
\begin{equation*}
  L= 2\ln\left(\frac{\sqrt{\vert s-m\vert }+\sqrt{\vert s+m\vert }}{\sqrt{2\vert m\vert}}\right)
\end{equation*}
and
\begin{equation}
  u(t,x) = \frac{s }{\cosh\left(L\right)} \cosh(x-st-x_{0}), \qquad |x-st-x_{0}| \le L.
  \label{eq:peakon_per}
\end{equation}
Note that for periodic peakons we have the identity $s = m \cosh(L)$, and so we could have equally well expressed
\eqref{eq:peakon_per} using $s$ and $m$.

Again, let us for simplicity assume that we are in the case $s > m>0$, such that $x_{0} \in \R$ gives the initial position of a trough of $u$. If we fix $\xi$ and define $w(t) = y(t,\xi)-st-x_0$, we can combine \eqref{eq:char} and \eqref{eq:peakon_per}
to obtain the differential equation
\begin{equation*}
  w'(t) = s \left( \frac{\cosh(w(t))}{\cosh(L)} -1  \right), \qquad w(0) = y_0(\xi) - x_0,
\end{equation*}
which has a Lipschitz-continuous right-hand side. Noting that $w=\pm L$ is a constant solution and that $w'(t)<0$ whenever $|w(t)|<L$, it is sufficient to consider the cases $w(0)=\pm L$ and $|w(0)|<L$. In particular, we obtain 
\begin{equation}
  \label{eq:peakon_per_char}
  \begin{aligned}
    w(t) &= - 2\artanh\left(\tanh\left(\frac{L}{2}\right) \tanh\left(\frac{st}{2}\tanh(L)-\artanh\left(\frac{\tanh\left(\frac{w(0)}{2}\right)}{\tanh\left(\frac{L}{2}\right)}\right)\right)\right)
  \end{aligned}
\end{equation}
which holds for $|w(0)| \le L$.
When translating this into the characteristics $y(t,\xi)$,
we must use the periodicity condition $y(t,\xi+2L) = y(t,\xi) + 2L$ outside this interval.
Note again that each peak follows a characteristic, which acts as asymptote for the characteristics in between:  
\begin{equation}
  \lim\limits_{t \to \pm \infty} (y(t,\xi) - st) = x_{0} \mp L, \qquad |y_0(\xi)-x_{0}| < L.
  \label{eq:peakon_per_asymp}
\end{equation}
This clustering behavior is demonstrated in Figure \ref{fig:peakon} for $y_0(\xi) = \xi$.
\begin{figure}
  \centering
  \begin{subfigure}[t]{0.495\textwidth}
    \includegraphics[width=\textwidth]{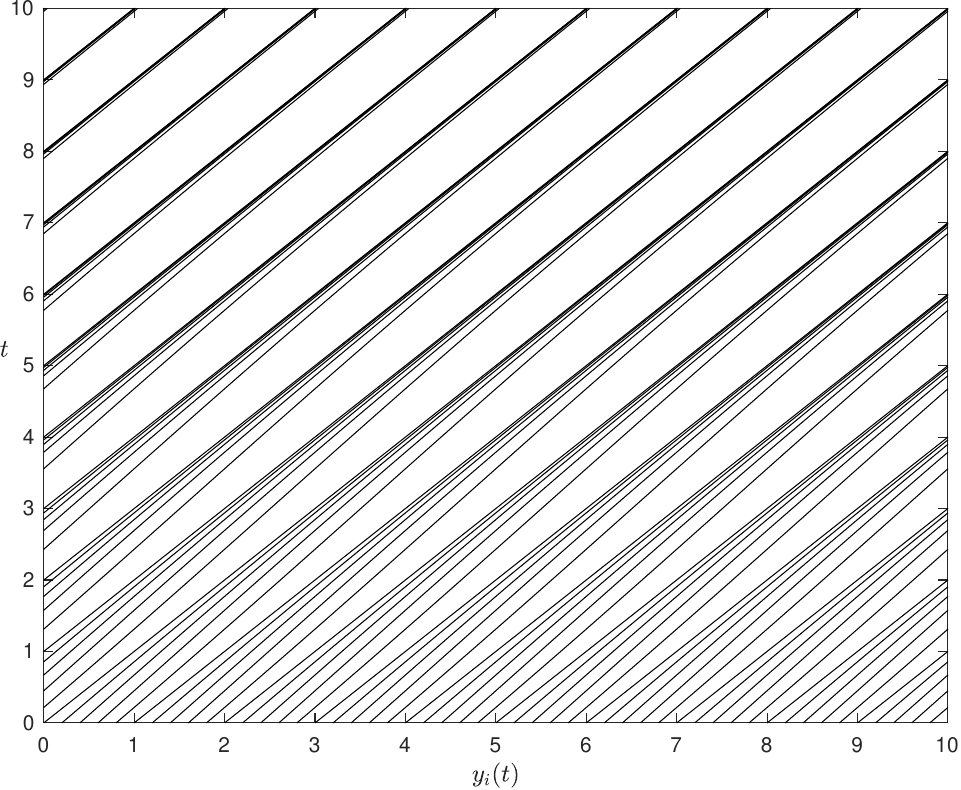}
    \caption{}
  \end{subfigure}
  \begin{subfigure}[t]{0.495\textwidth}
    \includegraphics[width=\textwidth]{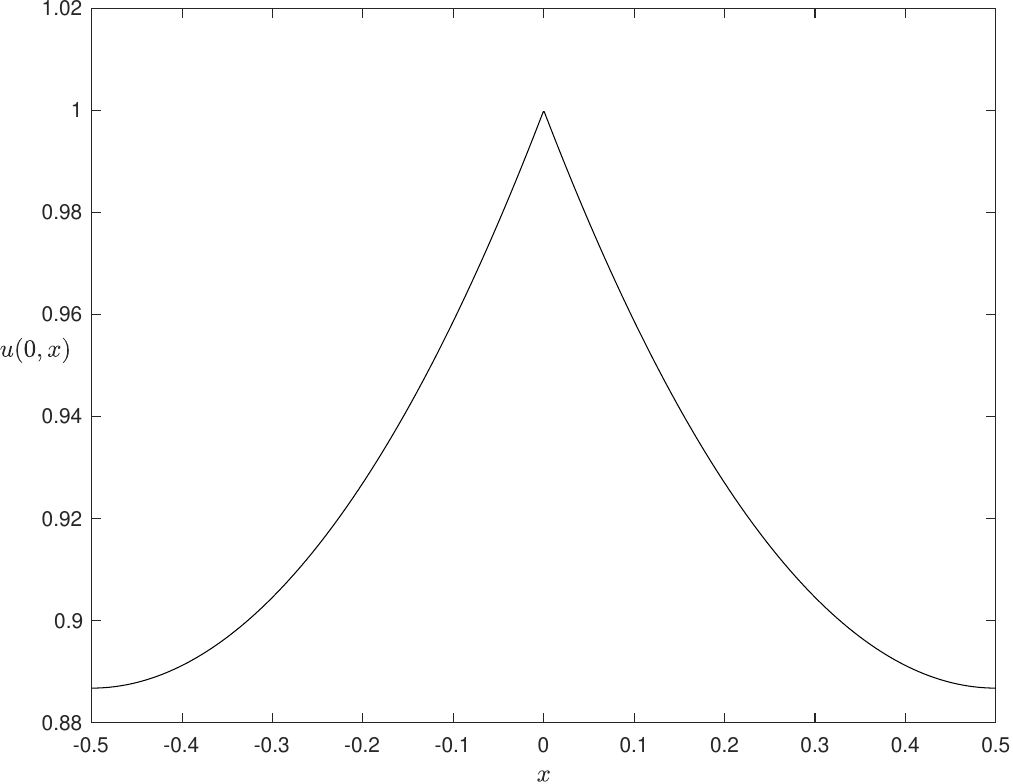}
    \caption{}
  \end{subfigure}
  \caption{Periodic peakon with parameters $x_0 = 0.5$, $s = 1$ and $2L = 1$, or alternatively $m = 1/\cosh(0.5)$. (\textsc{a}) The characteristics given by \eqref{eq:peakon_per_char} for $0 \le t \le 10$ and $\xi \in \{j/5\}_{j=-50}^{50}$. (\textsc{b}) The unitial profile for the periodic peakon given by \eqref{eq:peakon_per}.}
  \label{fig:peakon}
\end{figure}

\vspace{0.2cm}
As we have seen, the characteristics of peakons travel with a speed which is at most $s$, the speed of the peakon profile, and each peak follows a single characteristic with speed $s$.
Moreover, the characteristics of peakons never meet in finite time, which in particular means that there is no wave breaking.

\subsection{Cuspons}
The name {\it cuspon} comes from its {\it cusp}-shaped crest or trough, which, unlike the peakons, have an unbounded slope.
Indeed, in \cite{Lenells2005} it is shown that if the cusp is located in $x_c$, one has $\phiv(x) -\phiv(x_c) = \Ocal(\abs{x-x_c}^{2/3})$ as $x \to x_c$.
In analogy with the peakon, there are cuspons with decay as well as periodic ones.
\subsubsection{Cuspon with decay}
In the classification in \cite{Lenells2005}, cuspons with decay have parameters $z+m = 2m=s-M$ and either 
\begin{equation}\tag{f}
  m < s < M, \enspace m = \inf\limits_{x\in\R} \phiv(x), \enspace s = \max\limits_{x\in\R}\phiv(x), \enspace \phiv(x) \searrow m \text{ exponentially as } x \to \pm\infty,
\end{equation}
or 
\begin{equation}\tag{f'}
  m > s > M, \enspace m = \sup\limits_{x\in\R} \phiv(x), \enspace s = \min\limits_{x\in\R}\phiv(x), \enspace \phiv(x) \nearrow m \text{ exponentially as } x \to \pm\infty.
\end{equation}
In the following we denote cuspons by $\phi$ (instead of $\phiv$) and accordingly, \eqref{eq:dphi2} becomes
\begin{equation}\label{eq:dphi2Cdec}
  \phi_x^2 = \frac{(M-\phi) (\phi-m)^2}{s-\phi}.
\end{equation}
For simplicity we assume from now on that $s>m$ and since we are free to translate the cuspon as we wish, we take $\phi$ to have its cusp at the origin. It then follows that $\phi$ satisfies
\begin{subequations}\label{eq:phi_prop_dec}
  \begin{align}
    \phi(-x) &= \phi(x), \\
    \phi(0) &= s, \\
    \phi_x(x) &< 0 \enspace \text{for} \enspace x > 0, \\
    \lim\limits_{x\to\infty} \phi(x) &= m.
  \end{align}
\end{subequations}
Here we see that for any $\phi$ defined through \eqref{eq:dphi2Cdec} and \eqref{eq:phi_prop_dec} we have $m \le \phi(x) \le s$ for $x \in \R$.
Furthermore, $\phi_x(x)$ is finite for $x \neq 0$ and
\begin{equation}\label{eq:phi_break}
  \lim\limits_{x \to 0\mp} \phi_x(x) = \pm\infty \quad \text{and} \quad \lim\limits_{x\to0} \phi_x^2(x) = \infty.
\end{equation}
Hence, $u(t,x) = \phi(x-st)$ satisfies \eqref{eq:CH0} and has a well-defined
$u_x(t,x)$ for $x \neq st$ and
\begin{equation*}\label{eq:wb}
  \lim\limits_{x\to st}u_x^2(t,x) = \infty, \qquad t \in \R.
\end{equation*}
That is, for any time $t$, this solution features wave breaking at $x =st$.

These cuspons are studied in \cite{Grunert2016}, where it is found that, unlike the peak of a peakon, the cusp of the cuspon does
not follow any single characteristic.
Instead, the cusp moves faster than any of the characteristics to its left and right.
\subsubsection{Periodic cuspon}
In the classification in \cite{Lenells2005}, the periodic cuspons have parameters satisfying $z+m=s-M$ and either
\begin{equation}\tag{e}
  z < m < s < M, \enspace m = \min\limits_{x\in\R} \phiv(x), \enspace s = \max\limits_{x\in\R}\phiv(x),
\end{equation}
or 
\begin{equation}\tag{e'}
  z > m > s > M, \enspace m = \max\limits_{x\in\R} \phiv(x), \enspace s = \min\limits_{x\in\R}\phiv(x).
\end{equation}
Recall from before that we denote cuspons by $\phi$, so that in this case, \eqref{eq:dphi2} becomes
\begin{equation}\label{eq:dphi2Cper}
  \phi_x^2 = \frac{(M-\phi) (\phi-m)(\phi-z)}{s-\phi}.
\end{equation}
Assuming that the value $m$ is attained at $x=x_0$, the period $2L$ can be computed using the following relation,
which is derived from \eqref{eq:dphi2Cper}:
\begin{equation}\label{eq:L:cusp}
  \begin{aligned}
    L&=\int_{x_0}^{x_0+L} \sgn(\phi_x(x))\frac{\phi_x(x)\sqrt{\vert s-\phi(x)\vert}}{\sqrt{\vert M-\phi(x)\vert}\sqrt{\vert\phi(x)-m\vert}\sqrt{\vert \phi(x)-z\vert}}dx\\
    &=\sgn(s-m) \int_m^s \frac{\sqrt{\vert s-y\vert}}{\sqrt{\vert M-y\vert}\sqrt{\vert y-m\vert}\sqrt{\vert y-z\vert}}dy.
  \end{aligned}
\end{equation}

Let $2L$ be the period of the cuspon and assume from now on, for simplicity, that $s>m$. Then we can center the cusp at the origin to obtain the following properties for the $2L$-periodic function $\phi$ on the interval $[-L,L]$:
\begin{subequations}\label{eq:phi_prop_per}
  \begin{align}
    \phi(-x) &= \phi(x), \\
    \phi(0) &= s, \\
    \phi_x(x) &< 0 \enspace \text{for} \enspace 0 < x < L, \\
    \phi(L) &= m, \enspace \phi_x(L) = 0,
  \end{align}
\end{subequations}
where $ \phi_x(\pm L) = 0$ follows from property \ref{enum:tw1} since $\phi \in C^{\infty}$ away from the cusps.

\subsection{Stumpons}
The name \textit{stumpon} comes from the {\it stump}-shaped crests or troughs of these traveling waves.
In a sense, these solutions are constructed from cuspons: as stated in \ref{enum:tw2}, if one requires the parameter $a$ in \eqref{eq:tw_cond}
to satisfy
\begin{equation}\label{eq:a:stumpon}
  a = s^2 ,
\end{equation}
one can join cuspons at their cusps with a horizontal plateau of elevation $s$ equal to their velocity.
This produces a weak solution of \eqref{eq:CH0} with velocity $s$.
As cuspons are either decaying or periodic, this carries over to the stumpons as well.

\subsubsection{Stumpon with decay} 
Let the parameters $M$, $s$, and $m$ be as for the cuspons with decay, but with the additional constraint $a = s^2$.
In particular, assume that \eqref{eq:a:stumpon} holds for the cuspon $\phi$ given through  \eqref{eq:dphi2Cdec} and \eqref{eq:phi_prop_dec}, and define
\begin{equation}\label{eq:psi}
  \psi(x) = \begin{cases}
    \phi(\abs{x}-\ell), & \abs{x} \ge \ell, \\
    s, & \abs{x} < \ell
  \end{cases}
\end{equation}
for some $\ell > 0$. Then
\begin{equation}\label{eq:stumpon}
  u(t,x) = \psi(x-st)
\end{equation}
is a stumpon solution of \eqref{eq:CH0} with ``stumpwidth'' $2\ell$.
It follows from the considerations made for the cuspon solution that
$u(t,x)$ in \eqref{eq:stumpon} has a well-defined $u_x(t,x)$ for $x \neq st \pm \ell$, and for any time $t$ the solution features wave breaking at $x = st \pm \ell$.
Based on \eqref{eq:phi_prop_dec} one easily derives the following properties for $\psi$ in \eqref{eq:psi}:
\begin{subequations}\label{eq:psi_prop_dec}
  \begin{align}
    \psi(-x) &= \psi(x), \\
    \psi(x) &= s \enspace \text{for} \enspace \abs{x} \le \ell, \\
    \psi_x(x) &=  0 \enspace \text{for} \enspace x \in (0,\ell), \quad \psi_x(x) < 0 \enspace \text{for} \enspace x \in (\ell,\infty), \\
    \lim\limits_{x\to\infty} \psi(x) &= m.
  \end{align}
\end{subequations}

\subsubsection{Periodic stumpon}
In analogy with the previous section we now let the parameters $M$, $s$, $m$, and $z$ be as for the periodic cuspons, but again require $a = s^2$.
Let $\phi$ be the periodic cuspon given through \eqref{eq:dphi2Cper} and \eqref{eq:phi_prop_per}, and define the $2(L+\ell)$-periodic function $\psi$ by
\begin{equation}\label{eq:psiper}
  \psi(x) = \begin{cases}
    \phi(\abs{x}-\ell), & \ell\le\abs{x} \le \ell+L, \\
    s, & \abs{x} < \ell.
  \end{cases}
\end{equation}
Then 
\begin{equation}\label{eq:stumpon_per}
  u(t,x)=\psi(x-st)
\end{equation}
is a $2(L+\ell)$-periodic stumpon of \eqref{eq:CH0} with ``stumpwidth'' $2\ell$.
The corresponding periodic version of \eqref{eq:psi_prop_dec} is then given by
\begin{subequations}\label{eq:psi_prop_per}
  \begin{align}
    \psi(-x) &= \psi(x), \\
    \psi(x) &= s \enspace \text{for} \enspace \abs{x} \le \ell, \\
    \psi_x(x) &=  0 \enspace \text{for} \enspace x \in (0,\ell), \quad \psi_x(x) < 0 \enspace \text{for} \enspace x \in (\ell,L+\ell), \\
    \psi(L+\ell) &= m, \enspace \psi_x(L+\ell) = 0.
  \end{align}
\end{subequations}

\section{Conservative and non-conservative, periodic traveling waves}\label{s:cons:per}

The possibility of wave breaking for solutions of the Camassa--Holm equation introduces the non-uniqueness of weak solutions $u$, because one may remove some of the concentrated energy, and still retain a weak solution. However, one can construct a solution, which preserves the energy with respect to time, by imposing an additional constraint in form of an additional equation which has to be satisfied in the weak sense by the energy $\mu$. In the case of a
$2L$-periodic solution this system reads 
\begin{subequations}\label{eq:CHcons}
  \begin{align}
    u_t+uu_x& =-P_x \label{eq:CHcons:u} \\
    \mu_t+(u\mu)_x& = (u^3-mu^2+m^2u-2P(u-m))_x \label{eq:CHcons:mu} \\
    & =((u-m)^3+2m(u-m)^2+2m^2(u-m)+m^3-2P(u-m))_x, \notag
  \end{align}
\end{subequations}
where
\begin{equation}\label{eq:CHcons:P}
  P(t,x)=\frac14 \int_\R \e^{-\vert x-z\vert} (u^2+2mu-m^2)(t,z) \dee z +\frac14 \int_\R \e^{-\vert x-z\vert} \dee \mu(t,z).
\end{equation}
Thus every $2L$-periodic, conservative solution is given through a pair $(u,\mu)$ such that $(u(t), \mu(t))\in \Dcal$ for all $t\in \R$, where 
\begin{equation}\label{eq:D}
  \Dcal \coloneqq \left\{ (u,\mu) \mid u \in H^1_\text{per}(\R), \enspace \mu \in \Mcal^+_\text{per}, \enspace \mu_\text{ac} = ((u-m)^2+ u_x^2) \dee{x} \right\}.
\end{equation}
Here $m$ denotes a constant, $\Mcal^+_\text{per}$ the set of positive and periodic Radon measures on $\R$,
and $\mu_\text{ac}$ the absolutely continuous part of a measure $\mu$.
Note that we can apply the periodicity to replace the integral over the real line using the kernel $\frac12 e^{-\abs{x-z}}$ in \eqref{eq:CHcons:P}
with an integral over one period $2L$ using the periodized kernel, that is
\begin{equation*}
  P(t,x)=\frac14 \int_{-L}^L \frac{\cosh(\abs{x-z}-L)}{\sinh(L)} (u^2+2mu-m^2)(t,z) \dee z +\frac14 \int_\R \frac{\cosh(\abs{x-z}-L)}{\sinh(L)} \dee \mu(t,z),
\end{equation*}
cf. \cite[Equation (16)]{GalGru2021}.
This is in line with the kernel used in, e.g., \cite{ConMcK1999, HolRay2008}.

We elaborate a bit on the system \eqref{eq:CHcons}--\eqref{eq:CHcons:P}. First, we see that \eqref{eq:CHcons:u} is simply \eqref{eq:CHPa}.
Then, \eqref{eq:CHcons:P} is based on \eqref{eq:P}, but rewritten in order to incorporate the measure $\mu = \mu(t,x)$ from
\eqref{eq:D}.
The evolution equation \eqref{eq:CHcons:mu} is formally equivalent to \eqref{eq:claw} by replacing $\mu$ with $(u-m)^2 + u_x^2$
and using \eqref{eq:CHcons:u}.
We emphasize that in the case $\mu = \mu_\text{ac}$ then \eqref{eq:CHcons:P} reduces to \eqref{eq:P},
while \eqref{eq:CHcons:u} and \eqref{eq:CHcons:mu} is equivalent to \eqref{eq:CHPa} and \eqref{eq:claw}.
As for the constant $m$, it is not strictly necessary to introduce it here in the periodic setting,
and thereby ``complicate'' the equations somewhat by adding more terms.
However, by doing so we will alleviate the transition from studying the periodic case to the decaying case in
Section \ref{s:cons:dec}, as the proofs can be reiterated with only minimal changes.
There $m$ will represent the asymptotic value of the wave profiles $\phi$ and $\psi$, cf.\ Section \ref{s:travel}, and this ensures that $\mu_\text{ac}(\R)$
as defined in \eqref{eq:D} is bounded.
For the current periodic setting, we do not have this issue, as we only consider the energy
contained in a single period.

In the case of the real line, it has been proved that to every initial pair $(u_0,\mu_0)$ there exists a unique conservative solution and it can be derived using the method of characteristics, see \cite{BreCheZha2015}. Since this result has not been proven yet in the periodic case, we aim at showing that the periodic stumpon is not a weak, conservative solution, in the sense that it cannot be obtained by applying the method of characteristics to \eqref{eq:CHcons}. For an in-detail description of associating to every periodic initial data a conservative solution with the help of the method of characteristics, we refer to \cite{HolRay2008}.

Thereafter, we will investigate whether or not the numerical method developed in \cite{GalRay2021} and \cite{GalGru2021} approximates the weak, conservative solution with stumpon initial data as described by the method of characteristics in \cite{HolRay2008}.

\begin{thm}\label{thm:noncon}
  Periodic stumpons, unlike periodic cuspons, are not conservative solutions of the Camassa--Holm equation, since they do not satisfy the conservation law \eqref{eq:claw} in the weak sense.  
\end{thm}

\begin{proof}
  \begin{figure}
    \centering
    \includegraphics[width=0.5\textwidth]{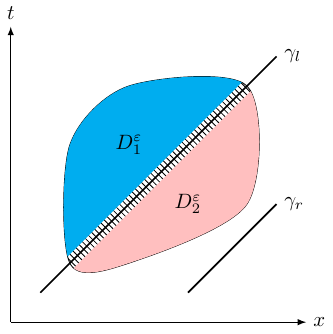}
    \caption{The leftmost ``gluing line'' $\gamma_l$ divides the region $D$ in two parts, $D_1$ and $D_2$. For each $i =1,2$ and a given $\varepsilon > 0$ we have an open region $D_i^\varepsilon \subset D_i$ with distance $\varepsilon$ to $\gamma_l$, where the solution is smooth.}
    \label{fig:RH}
  \end{figure}
  The periodic stumpon with period $2(L+\ell)$ is given by \eqref{eq:stumpon_per} and hence there are two types of gluing lines. The ones where the plateau is to the right and the ones where the plateau is to the left. Since the stumpon is periodic it suffices to study the gluing lines starting inside the interval $[-L-\ell, L+\ell]$. These are the gluing lines 
  \begin{equation}
    \gamma_l \colon t\mapsto -\ell+st \quad \text{ and } \quad \gamma_r \colon t\mapsto \ell+st.
  \end{equation}
  Based on an argument inspired by the derivation of the Rankine--Hugoniot condition for conservation laws we aim to show that any stumpon does not satisfy \eqref{eq:claw} in the weak sense. 
  
  Given a point on $\gamma_l$, which we denote $(\bar t,\gamma_l(\bar t))$, choose a neighborhood $D$ around $(\bar t, \gamma_l(\bar t))$, which does not contain parts of $\gamma_r$ nor any other gluing line. Furthermore divide $D$ into three non-intersecting parts $D=D_1\cup \gamma_l\vert_D\cup D_2$, as indicated in Figure~\ref{fig:RH}, and let 
  \begin{equation*}
    D_i^\varepsilon=\{(t,x)\in D_i\mid \text{dist}((t,x),\gamma_l)>\varepsilon\} \quad \text{ for } i \in \{1,2\}.
  \end{equation*}
  Pick any test function $\Phi(t,x)$ whose support lies entirely inside $D$. If the stumpon would satisfy \eqref{eq:claw}, then 
  \begin{align*}
    0&=\iint_D ((u^2+u_x^2)\Phi_t+(u(u^2+u_x^2)-u^3+2Pu)\Phi_x)(t,x)\dee x \dee t \\
    & = \lim_{\varepsilon\to 0}\iint_{D_1^\varepsilon} ((u^2+u_x^2)\Phi_t+(u(u^2+u_x^2)-u^3+2Pu)\Phi_x)(t,x)\dee x \dee t\\
    & \quad +\lim_{\varepsilon\to 0}\iint_{D_2^\varepsilon} ((u^2+u_x^2)\Phi_t+(u(u^2+u_x^2)-u^3+2Pu)\Phi_x)(t,x)\dee x \dee t.
  \end{align*}
  By assumption, for any $\varepsilon>0$ the function $u$ is smooth in $D_1^\varepsilon\cup D_2^\varepsilon$. In particular, one has that \eqref{eq:claw} holds for every point $(t,x)$ in $D_1^\varepsilon\cup D_2^\varepsilon$.
  Let us introduce the notation 
  \begin{equation*}
    \gamma_i^\varepsilon =\{ (t,x)\in D_i\mid \text{dist}((t,x),\gamma_l)=\varepsilon\}
  \end{equation*}
  and 
  \begin{equation*}
    I_i^\varepsilon=\{t\in \mathbb{\R}\mid (t, \gamma_i(t))\in D_i^\varepsilon\}, \qquad I_l = \{t\in \mathbb{\R}\mid (t, \gamma_l(t))\in D\}. 
  \end{equation*}
  We can thus write
  \begin{align*}
    &  \iint_{D_1^\varepsilon} ((u^2+u_x^2)\Phi_t+(u(u^2+u_x^2)-u^3+2Pu)\Phi_x)(t,x)\dee x \dee t\\
    &  \qquad = \iint_{D_1^\varepsilon} (((u^2+u_x^2)\Phi)_t+((u(u^2+u_x^2)-u^3+2Pu)\Phi)_x)(t,x)\dee x \dee t\\
    &  \qquad =\int_{\partial D_1^\varepsilon} -(u^2+u_x^2)\Phi \dee x +(u(u^2+u_x^2)-u^3+2Pu)\Phi \dee t\\
    & \qquad = \int_{I_1^\varepsilon} ((u^2+u_x^2)(u-s)-u^3+2Pu)\Phi (t,\gamma_1^\varepsilon(t)) \dee t\\
    & \qquad = \int_{I_1^\varepsilon} (u_x^2(u-s)-u^2s+2Pu)\Phi (t,\gamma_1^\varepsilon(t))\dee t,
  \end{align*}
  where we have used Green's theorem, $(\gamma_1^\varepsilon)'(t)=s$, and that $\Phi$ vanishes on the part of the boundary given by
  $\partial D_i^\varepsilon \setminus \gamma_i^\varepsilon$.
  
  Note that $\gamma_1^\varepsilon(t) = \gamma_l(t) -\varepsilon\sqrt{1+s^2}$ for $t \in I_1^\varepsilon$, and similarly
  $\gamma_2^\varepsilon(t) = \gamma_l(t) +\varepsilon\sqrt{1+s^2}$ for $t \in I_2^\varepsilon$.
  From \eqref{eq:dphi2} it follows that $u(t,x)=\psi(x-st)$ satisfies
  \begin{equation*}
    u_x^2(u-s) = 
    (\psi')^2(\psi-s) =-(M- \psi)(\psi-m)(\psi-z).
  \end{equation*}
  on $D_1^\varepsilon$.
  Moreover, recall that $\psi = s$ on the plateau, such that the continuity of $u(t,x)$ and \eqref{eq:Ptw} yield $u(t,\gamma_l(t)) = s$, $P(t,\gamma_l(t)) = s^2$.
  Thus we end up with
  \begin{align*}
    &  \lim_{\varepsilon \to 0}\iint_{D_1^\varepsilon} ((u^2+u_x^2)\Phi_t+(u(u^2+u_x^2)-u^3+2Pu)\Phi_x)(t,x)\dee x \dee t\\ 
    & \qquad =\left( -(M-s)(s-m)(s-z) +s^3 \right) \int_{I_l} \Phi(t, \gamma_l(t))\dee t.
  \end{align*}
  Following the same lines for the integral over $D_2^\varepsilon$ and recalling that $u(t,x)=s$ inside $D_2^\varepsilon$, we obtain 
  \begin{equation*}
    \lim_{\varepsilon \to 0}\iint_{D_2^\varepsilon} ((u^2+u_x^2)\Phi_t+(u(u^2+u_x^2)-u^3+2Pu)\Phi_x)(t,x)\dee x \dee t=-s^3 \int_{I_l} \Phi(t, \gamma_l(t))\dee t,
  \end{equation*}
  where the negative sign comes from traversing $\gamma_l$ in the opposite direction from before.
  Combining the above expressions, we find that for \eqref{eq:claw} to hold weakly, we must have
  \begin{equation*}
    0= (M-s)(s-m)(s-z)\int_{I_l}\Phi(t, \gamma_l(t))\dee t.
  \end{equation*}
  Recall that the above equality must hold for any test function $\Phi(t,x)$ whose support lies entirely inside $D$, and hence we must have $(M-s)(s-m)(s-z)=0$, which contradicts the assumptions on the parameters $m$, $s$, and $M$
  for periodic cuspons.
  A completely analogous argument shows that the same issue arises around the gluing line $\gamma_r$.
\end{proof}

Thus, a stumpon is not a weak, conservative traveling wave solution in the sense of \cite{HolRay2008}. One therefore expects that any numerical method based on the method of characteristics introduced therein will not yield the traveling stumpon as a weak, conservative solution with stumpon initial data. On the other hand, carrying out the above analysis for a cusp, one obtains that those satisfy \eqref{eq:claw} and therefore they can be approximated using numerical methods based on \cite{HolRay2008}.

Next, we illustrate Theorem~\ref{thm:noncon} by applying the numerical method presented in \cite{GalGru2021}, which again is based on the discretization in Lagrangian variables derived in \cite{GalRay2021}.
In fact, it was these resulting figures which led us to investigate the stumpon in detail, as they did not produce the expected result, namely the stumpon solution.

For the cuspon we have chosen the parameter values $m = 0$, $s = 1$, and $M = (1+\sqrt{5})/2$,
and plugging these values into \eqref{eq:a:tw} verifies that \eqref{eq:a:stumpon} is satisfied.
Using the method presented in \cite{KalLen2005} we have computed a reference cuspon with period $2L_\phi \approx 2.969$.
To avoid overloading the illustrations of the characteristics we have chosen to discretize with $N = 512$ discrete labels,
and we have run the scheme until $T = 4L_\phi$, that is, two periods.
Moreover, for simplicity we have chosen the initial labeling $y(0,\xi) = \xi$, which is strictly speaking not a valid relabeling function for the Lagrangian variables.
However, since the singularity at the cusp is a single point, and we are interpolating by following initially equally spaced characteristics $y(0,\xi_i)$, 
this is permissible for the discretization due to relabeling.
\begin{figure}
  \centering
  \begin{subfigure}[t]{0.495\textwidth}
    \includegraphics[width=\textwidth]{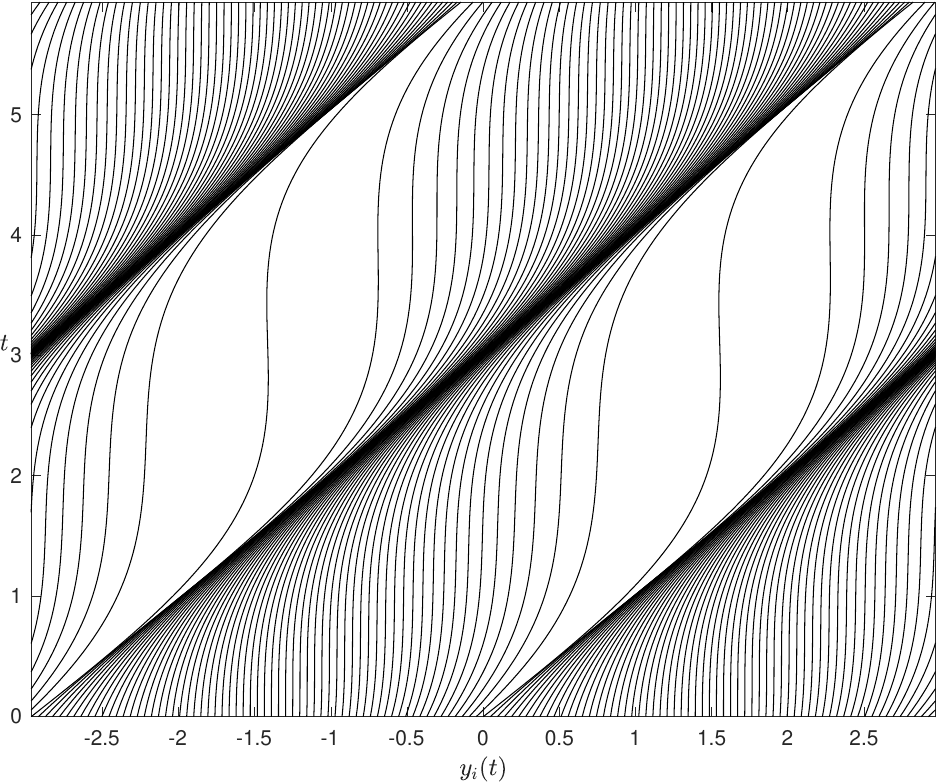}
    \caption{}
  \end{subfigure}
  \begin{subfigure}[t]{0.495\textwidth}
    \includegraphics[width=\textwidth]{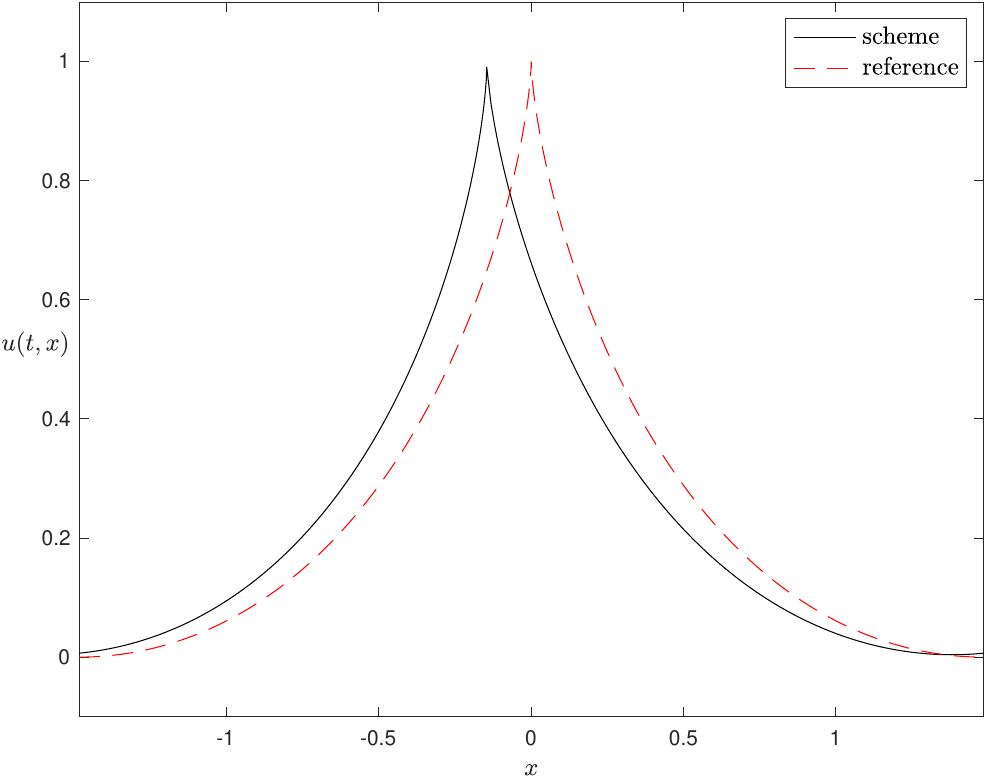}
    \caption{}
  \end{subfigure}
  \caption{Numerical results for cuspon initial data with $ m = 0$, $s = 1$, $M = (1+\sqrt{5})/2$, and period $2L_\phi \approx 2.969$. (\textsc{a}) Numerically computed $y(t,\xi)$ for $y(0,\xi) = \xi$, showing 64 of the total 512 characteristics. (\textsc{b}) The numerical and reference solutions at $T = 4L_\phi$, i.e., two periods.}
  \label{fig:cuspon}
\end{figure}
The characteristics displayed in Figure \ref{fig:cuspon} illustrate nicely how the cusp,
located along the line where the characteristics are most dense, moves faster than the characteristics.
This is quite contrary to the characteristics of the periodic peakon in Figure \ref{fig:peakon},
where the characteristics cluster in front of the peaks.
Figure \ref{fig:cuspon} also shows the numerical and reference solution at time $T = 4L_\phi$,
where we see that the numerical solution manages to uphold a nicely cusped profile.
The computed solution is lagging slightly behind the reference solution, but this phase error is reduced by increasing the number of discrete labels $N$.

The parameters for the cuspon have been chosen to satisfy the constraint $a = s^2$, which means that we can add a plateau of height $s = 1$ and length $4-2L_\phi$ at the cusp to obtain a valid stumpon. Adding this plateau, we obtain a stumpon with period $2L_\psi=4$,
for which we run our numerical scheme with the same discretization parameters $N = 512$ and $T = 4L_\psi = 8$.
\begin{figure}
  \centering
  \begin{subfigure}[t]{0.495\textwidth}
    \includegraphics[width=\textwidth]{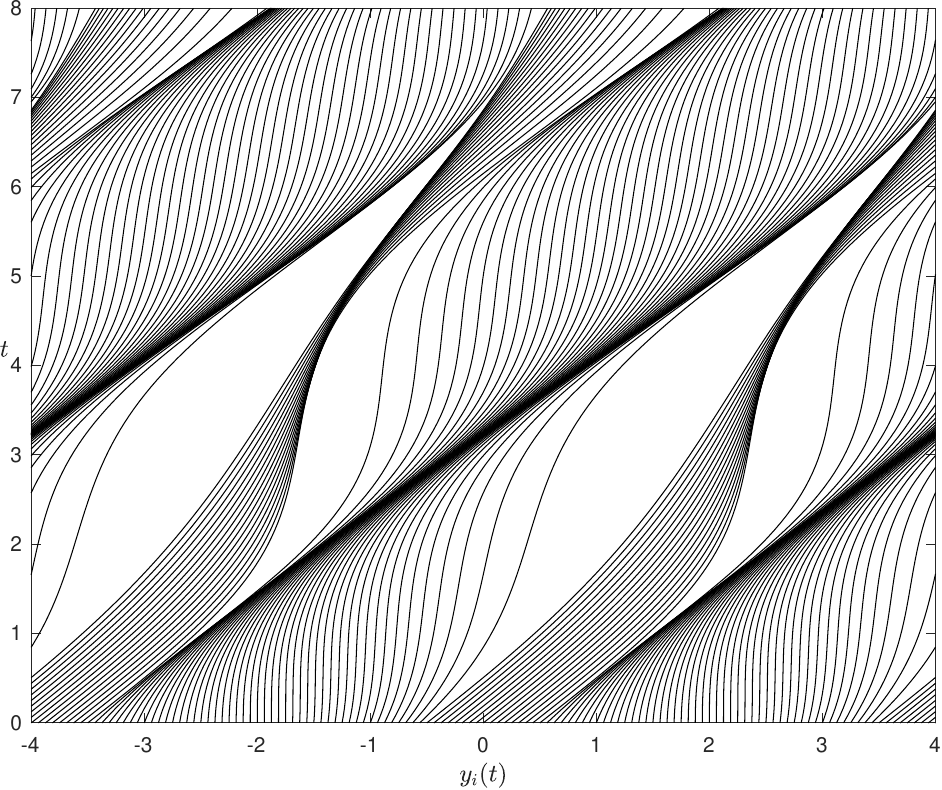}
    \caption{}
  \end{subfigure}
  \begin{subfigure}[t]{0.495\textwidth}
    \includegraphics[width=\textwidth]{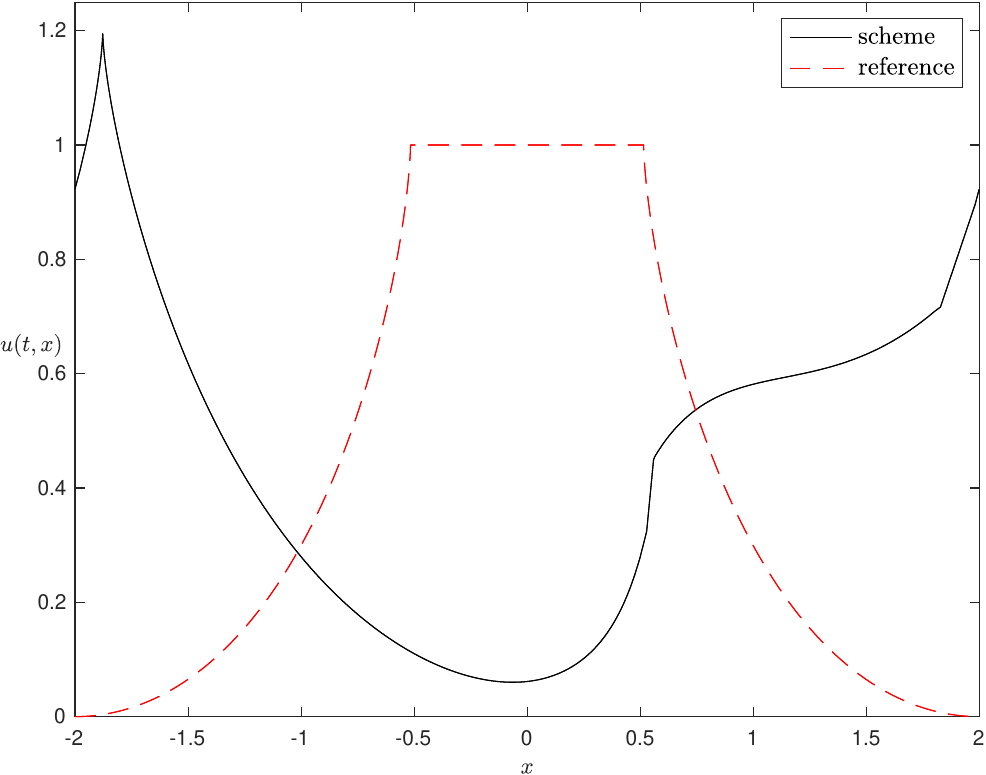}
    \caption{}
  \end{subfigure}
  \caption{Numerical results for stumpon inital data with $m = 0$, $s = 1$, $M = (1+\sqrt{5})/2$, and period $2L_\psi = 4$. (\textsc{a}) Numerically computed $y(t,\xi)$ for $y(0,\xi) = \xi$, showing 64 of the total 512 characteristics. (\textsc{b}) The numerical and reference solutions at $T = 8$, i.e., two periods.}
  \label{fig:stumpon1}
\end{figure}
This produces a very different result to that of the previous example, where the initial wave profile is not preserved at all.
Indeed, we observe that a cusped crest emerges from the front of the stump, leaving the rest of it behind.
On the other hand, the rest of the plateau falls down and collapses into a temporary cusped trough,
which then is overtaken by the crest at $t \approx 7$.
This behavior is also suggested by the associated characteristics displayed in Figure \ref{fig:stumpon1},
where the densely packed lines indicate the presence of a cusp.
From the slope of the densely packed characteristics we also see that the crest has a speed slightly greater than 1 before colliding with the trough.
Figure \ref{fig:stumpon1} also shows the numerical and reference solution at time $T = 8$,
where the reference stumpon has returned to its initial position, while the taller crest has just reemerged from the impact with the trough.

At this stage, there can only be two possible explanations for this, in our opinion, unexpected behavior:
Either there is a deficiency in the numerical method, leading to an ill-behaved approximation of the conservative solution. Alternatively, the numerical method produces a good approximation, which however contradicts our intuition. We therefore take a closer look at what we believe is the most prominent, yet unexpected feature of the numerical solution for small $t>0$: The ``collapse'' of the plateau. This is illustrated in Figure \ref{fig:stumpon2}, which shows the initial stumpon profile at $t = 0$ and compares the numerical and reference solutions
at $t = T/2$ and $t = T$ for $T = 4-2L_\phi$.
\begin{figure}
  \centering
  \includegraphics[width=0.75\textwidth]{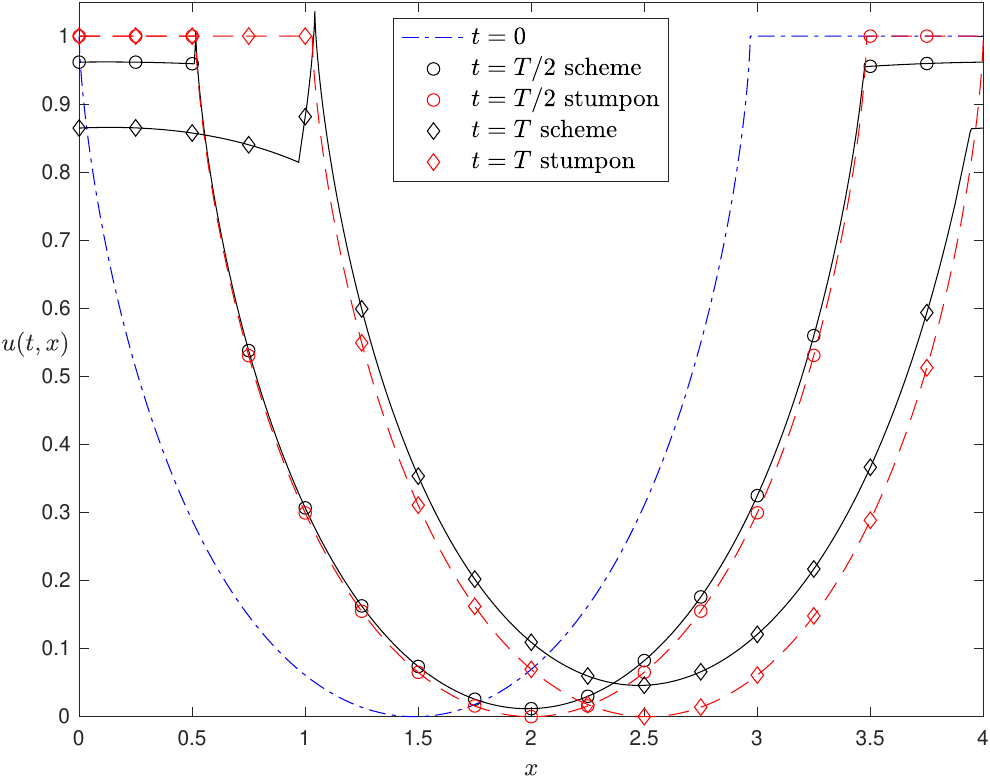}
  \caption{A closer look at the numerical solution from Figure \ref{fig:stumpon1}. Here the initial profile is shifted a distance $2-L_\phi$ to the left, and the reference and numerical solutions are shown at the times $t = 0$, $t = T/2$ and $t = T$ for $T = 4-2L_\phi$.}
  \label{fig:stumpon2}
\end{figure}
Here we clearly see how the plateau falls from the very beginning, while the cusped crest emerges to the right of the stump.
In order to justify these numerical results, we will prove the following result.
\begin{prp}\label{prp:platfall}
  Let $\psi$ be the periodic wave profile of a traveling stumpon with a plateau of positive width. Then, for the conservative solution of the Camassa--Holm equation with initial data $u_0(x) = \psi(x)$, the plateau will not persist.
\end{prp}
In particular, the plateau will descend for the case $z < m < s < M$, while it ascends for $z > m > s > M$.

To prove Proposition~\ref{prp:platfall} and to understand the numerical results obtained for the stumpon initial data, we need to introduce the method of characteristics from \cite{HolRay2008}.

Due to the possible concentration of energy at time $t=0$, for some general initial data $(u_0,\mu_0)\in \Dcal$, one cannot in general pick the initial characteristic $y(0,\xi)=\xi$. One possible mapping from Eulerian to Lagrangian coordinates is contained in the following definition, which is a slight modification of \cite[Theorem 3.8]{HolRay2008}.
\begin{dfn}\label{def:EtoL:per}
  For any $(u,\mu)\in \Dcal$ with period $2L$, let $(y,U,H)$ on $[0, 2(L+E)]$ be given by
  \begin{subequations}\label{eq:Lag:per}
    \begin{align}
      2E&=\mu([0,2L)),\\ \label{eq:Lag:per:y}
      y(\xi)&=\sup\{ x \in [0,2L) \mid x+\mu([0,x))<\xi\},\\ \label{eq:Lag:per:U}
      U(\xi)&=u(y(\xi)), \\ \label{eq:Lag:per:H}
      H(\xi)&= \xi-y(\xi),
    \end{align}
  \end{subequations}
  and extend $(y,U,H)$ to all of $\R$ by setting 
  \begin{equation}\label{per:cond}
    y(\xi+2(L+E))=y(\xi)+2L, \enspace U(\xi+2(L+E))=U(\xi), \enspace \text{and}\enspace H(\xi+2(L+E))=H(\xi)+2E.
  \end{equation}
\end{dfn}
An advantage of this mapping is that it admits colliding characteristics, which correspond to wave breaking, and therefore allows for a rarefaction-like behavior of the characteristics as time evolves. The evolution equations yielding conservative solutions are rooted in \eqref{eq:CHcons} and \eqref{eq:CHcons:P}, and are given by
\begin{subequations}\label{eq:evol:per}
  \begin{align}
    y_t&=U,\\
    U_t&=-Q,\\
    H_t& = U^3-mU^2+m^2U-2P(U-m)\\
    &=(U-m)^3+2m(U-m)^2+2m^2(U-m)+m^3-2P(U-m), \notag
  \end{align}
\end{subequations}
where 
\begin{subequations}
  \begin{align}\label{eq:Plagper}
    P(t,\xi)&=\frac{1}{4\sinh(L)}\int_{-L-E}^{L+E} \cosh(L-\vert y(t,\xi)-y(t,\eta)\vert) ((U^2+2mU-m^2)y_\xi+H_\xi)(t,\eta) \dee \eta, \\ \label{eq:Qper}
    Q(t,\xi)& =-\frac{1}{4\sinh(L)}\int_{-L-E}^{L+E} \sgn(\xi-\eta)\sinh(L-\vert y(t,\xi)-y(t,\eta)\vert) ((U^2+2mU-m^2)y_\xi+H_\xi)(t,\eta) \dee \eta 
  \end{align}
\end{subequations}
on $[-L-E, L+E]$. The solution is then extended to all of $\R$ using \eqref{per:cond}, while $P$ and $Q$ are $2(L+E)$-periodic.

In order to return to Eulerian coordinates, one can apply the following mapping, which is taken from \cite[Theorem 3.11]{HolRay2008}.
\begin{dfn}\label{def:LtoE:per}
  Given $(y,U,H)$, define $(u,\mu)\in\Dcal$ as follows
  \begin{subequations}
    \begin{align}
      u(x)&=U(\xi) \text{ for any } \xi \text{ such that }x=y(\xi),\\
      \mu&= y_{\#}(H_\xi \dee \xi).
    \end{align}
  \end{subequations}
\end{dfn}

When studying the conservative solution for periodic stumpon initial data, we will make use of the fact that the stumpon initial data $(u_{\psi,0}, \mu_{\psi,0})=(\psi, ((\psi-m)^2+(\psi')^2)\dee x)$ is closely related to the cuspon initial data $(u_{\phi,0}, \mu_{\phi,0})=(\phi, ((\phi-m)^2+(\phi')^2)\dee x)$.
We emphasize that in both cases, the initial measures are absolutely continuous, which we remember simplifies
the conservative evolution equations \eqref{eq:CHcons} and \eqref{eq:CHcons:P}.
In Lagrangian coordinates we will denote the solutions corresponding to stumpon and cuspon initial data by
$(y_\psi, U_\psi, H_\psi)$ and $(y_\phi, U_\phi, H_\phi)$, respectively. 

We recall from Section \ref{s:travel} that, without loss of generality, we chose $\phi$ and $\psi$ to be symmetric around the origin and such that $\phi(0)=\psi(0)=s$. By definition one has the following relations between the corresponding half-periods and half-energies
\begin{equation*}
  L_\psi=L_\phi+\ell \quad \text{ and }\quad E_\psi=E_\phi+ (s-m)^2\ell.
\end{equation*}
On the other hand, the symmetry assumption yields
\begin{align*}
  \psi(x-\ell)&=\phi(x), \quad \mu_{\psi,0}(( x-\ell,0)) =\mu_{\phi,0}((x,0))+(s-m)^2\ell \quad \text{ for } -L_\phi\leq x<0,\\
  \psi( x+\ell)&=  \phi(x), \quad \mu_{\psi,0}((0, x+\ell))=\mu_{\phi,0}((0,x))+(s-m)^2\ell\quad \text{ for } 0\leq x\leq L_\phi.
\end{align*}
From definition \eqref{eq:Lag:per:y} it follows that  $y_\phi(0,0)=0$, and 
\begin{align} \label{eq:shift:y}
  \begin{aligned}
    y_\psi(0,\xi-(1+(s-m)^2)\ell)&=y_\phi(0,\xi)-\ell \quad \text{ for } -L_\phi-E_\phi \leq\xi\leq 0 , \\
    y_\psi(0,\xi+(1+(s-m)^2)\ell)&=y_\phi(0,\xi)+\ell \quad \text{ for } 0\leq \xi\leq L_\phi+E_\phi.
  \end{aligned}
\end{align}
Furthermore, recalling \eqref{eq:psi}, \eqref{eq:Lag:per:H}, and \eqref{eq:Lag:per:U} one has $H_\phi(0,0) = 0$,
as well as
\begin{subequations}\label{eq:shift:UH}
  \begin{align}
    U_\psi(0,\xi-(1+(s-m)^2)\ell)&=U_\phi(0,\xi)\\
    H_\psi(0,\xi-(1+(s-m)^2)\ell)&=H_\phi(0,\xi)- (s-m)^2\ell
  \end{align}
  for $-L_\phi-E_\phi \leq\xi \leq 0$ and  
  \begin{align}
    U_\psi(0,\xi+(1+(s-m)^2)\ell)&=U_\phi(0,\xi), \\
    H_\psi(0, \xi+(1+(s-m)^2)\ell)&=H_\phi(0,\xi) +(s-m)^2\ell
  \end{align}
\end{subequations}
for $0 \leq \xi\leq L_\phi+E_\phi$,
while on the plateau
\begin{equation}\label{eq:plat:UH}
  U_{\psi}(0,\xi)=s, \quad H_\psi(0,\xi)=\frac{(s-m)^2}{1+(s-m)^2}\xi \quad \text{ for } \xi\in (-(1+(s-m)^2)\ell, (1+(s-m)^2)\ell).
\end{equation}
In Eulerian coordinates, \eqref{eq:Ptw} yields 
\begin{equation*}
  P_\psi(0,x)=s^2-\frac12 (\psi(x)-s)^2,
\end{equation*}
which implies that $P_\psi(0,x) = s^2$ and $P_{\psi,x}(0,x)=0$ on the plateau. Furthermore, 
\begin{align*}
  P_\psi(0, x-\ell)&=P_\phi(0,x), \quad P_{\psi,x}(0,x-\ell)=P_{\phi,x}(0,x) \quad \text{ for } -L_\phi\leq x\leq 0, \\
  P_\psi(0,x+\ell)&= P_\phi(0,x), \quad P_{\psi,x}(0, x+\ell)=P_{\phi,x}(0,x) \quad \text{ for } 0\leq x\leq L_\phi.
\end{align*}
In Lagrangian coordinates this reads
\begin{subequations} \label{eq:shift:PQ}
  \begin{align}
    P_\psi(0,\xi-(1+(s-m)^2)\ell)&=P_\phi(0,\xi),\\
    Q_\psi(0,\xi-(1+(s-m)^2)\ell)&=Q_\phi(0,\xi) 
  \end{align}
  for $-L_\phi-E_\phi \leq \xi\leq 0$ and
  \begin{align}
    P_\psi(0,\xi+(1+(s-m)^2)\ell)&=P_\phi(0,\xi), \\
    Q_\psi(0,\xi+(1+(s-m)^2)\ell)&=Q_\phi(0,\xi)    
  \end{align}
\end{subequations}
for $0\leq \xi\leq L_\phi+E_\phi$,
while on the plateau we have
\begin{equation}\label{eq:plat:PQ}
  P_\psi(0,\xi)=s^2, \quad Q_\psi(0,\xi)=0 \quad \text{ for } \xi\in (-(1+(s-m)^2)\ell,(1+(s-m)^2)\ell).
\end{equation}
Combining the evolution equations \eqref{eq:evol:per} with \eqref{eq:plat:UH} and \eqref{eq:plat:PQ} we obtain
\begin{equation*}
  U_{\psi,t}(0,\cdot)=0, \quad H_{\psi,t}(0,\cdot)=-s^3+ms^2+m^2s \quad \text{ for } \xi\in (-(1+(s-m)^2)\ell, (1+(s-m)^2)\ell).
\end{equation*}
Therefore, in order to say something about the evolution of $U_{\psi}$ on the plateau
it is necessary to compute $U_{\psi,tt}(0,\cdot)=-Q_{\psi,t}(0,\cdot)$ there.

Following the proof of the existence and continuity of $Q_t$ for the cuspon with decay in \cite{Grunert2016} one can show, using \eqref{eq:Qper}, that $Q_t$ exists and is given by
\begin{align*}
  Q_t(t,\xi)&=\frac{1}{4\sinh(L)} \int_{-L-E}^{L+E}(U(t,\xi)-U(t,\eta))\cosh( L-\vert y(t,\xi)-y(t,\eta)\vert )\\
  & \qquad \qquad \qquad \qquad \qquad \qquad \qquad \qquad \times ((U^2+2mU-m^2)y_\xi+H_\xi)(t,\eta)\dee\eta\\
  & \quad -\frac{1}{4\sinh(L)}\int_{-L-E}^{L+E}\sgn(\xi-\eta) \sinh(L-\vert y(t,\xi)-y(t,\eta)\vert) (4U^2U_\xi-4UQy_\xi-2PU_\xi)(t,\eta) \dee \eta.
\end{align*}
In particular, by \eqref{eq:CHPb} and continuity, the periodic cuspon solution $\phi$ satisfies
\begin{equation*}
  Q_{\phi,t}(0,0)=\frac12 (M-s)(s-m)(s-z)\not =0.
\end{equation*}
Our aim is to use symmetries and shift identities to derive a relation between $Q_{\phi,t}(0,0)$ and $Q_{\psi,t}(0, \xi)$ for $\xi\in (-(1+(s-m)^2)\ell, (1+(s-m)^2)\ell)$.
This will in turn allow us to compute $Q_{\psi,t}(0,\xi)$ for all points on the plateau. 
For the cuspon we have 
\begin{equation}\label{eq:Qt0}
  \begin{aligned}
    Q_{\phi,t}(0,0)& =\frac{1}{4\sinh(L_\phi)} \int_{-L_\phi-E_\phi}^{L_\phi+E_\phi}(U_\phi(0,0)-U_\phi(0,\eta))\cosh( L_\phi-\vert y_\phi(0,\eta)\vert )\\
    &\qquad \qquad \qquad \qquad \qquad \qquad \qquad \qquad \times ((U_\phi^2+2mU_\phi-m^2)y_{\phi,\xi}+H_{\phi,\xi})(0,\eta)\dee\eta\\
    & \quad +\frac{1}{4\sinh(L_\phi)}\int_{-L_\phi-E_\phi}^{L_\phi+E_\phi}\sgn(\eta) \sinh(L_\phi-\vert y_\phi(0,\eta)\vert)\\
    &\qquad \qquad \qquad \qquad \qquad \qquad \qquad \qquad \times  (4U_\phi^2U_{\phi,\xi}-4U_\phi Q_\phi y_{\phi,\xi}-2P_\phi U_{\phi,\xi})(0,\eta) \dee\eta,
  \end{aligned}
\end{equation}
which can also be expressed using only integrals from $-L_\phi-E_\phi$ to $0$ or from $0$ to $L_\phi+E_\phi$ due to some symmetry properties outlined below. Indeed, note that we have  
\begin{equation*}
  \mu_{\phi,0}((-x,0))=\mu_{\phi,0}((0,x)) \quad \text{ for } x > 0.
\end{equation*} 
Furthermore, we can, in the case of a periodic cuspon $\phi$, in \eqref{eq:Lag:per}
replace $x\in [0,2L_\phi)$ with $x\in \R$  using the convention
\begin{equation*}
  \mu_{\phi,0}((0,x)) = -\mu_{\phi,0}((x,0)) \quad \text{ for } x < 0,
\end{equation*}
and still obtain the same initial data $(y_\phi(0,\cdot), U_\phi(0,\cdot), H_\phi(0,\cdot))$.
In particular, for all $k\in [-L_\phi-E_\phi, L_\phi+E_\phi]$ it follows from \eqref{eq:Lag:per:y} that
\begin{equation*}
  k=y_\phi(0,k)+\mu_{\phi,0}((0,y_\phi(0,k))),
\end{equation*} 
and \begin{equation*}
  k=-(-k)=-(y_\phi(0,-k)+\mu_{\phi,0}((0,y_\phi(0,-k))))=-y_\phi(0,-k)+\mu_{\phi,0}((0,-y_\phi(0,-k)) ).
\end{equation*}
Moreover, for $k>0$, we have $y_\phi(0,-k)<0$ and $y_\phi(0,k)>0$, while the function $x+\mu_{\phi,0}((0,x))$ is strictly increasing, which implies for all $k\in [-L_\phi-E_\phi, L_\phi+E_\phi]$ that
\begin{subequations}\label{eq:sym}
  \begin{equation}\label{eq:sym:y}
    y_\phi(0,-k) = -y_\phi(0,k).
  \end{equation}
  Combining \eqref{eq:sym:y} with \eqref{eq:Lag:per:U}, \eqref{eq:Lag:per:H} we find
  \begin{equation}\label{eq:sym:UH}
    U_\phi(0,-k)=U_\phi(0,k), \qquad H_\phi(0,-k)=-H_\phi(0,k).
  \end{equation}
  Finally, \eqref{eq:sym:y}, \eqref{eq:sym:UH}, and the symmetry properties inherited by their derivatives
  can be used in \eqref{eq:Plagper} and \eqref{eq:Qper} to obtain
  \begin{equation}\label{eq:sym:PQ}
    P_\phi(0,-k)=P_\phi(0,k), \qquad Q_\phi(0,-k)=-Q_\phi(0,k).
  \end{equation}
\end{subequations}
The symmetries in \eqref{eq:sym} can then be applied to rewrite \eqref{eq:Qt0} as
\begin{align*}
  Q_{\phi,t}(0,0)& =\frac{1}{2\sinh(L_\phi)} \int_{-L_\phi-E_\phi}^{0}(U_\phi(0,0)-U_\phi(0,\eta))\cosh( L_\phi-\vert y_\phi(0,\eta)\vert )\\
  &\qquad \qquad \qquad \qquad \qquad \qquad \qquad \qquad \times ((U_\phi^2+2mU_\phi-m^2)y_{\phi,\xi}+H_{\phi,\xi})(0,\eta)\dee \eta\\
  & \quad -\frac{1}{2\sinh(L_\phi)}\int_{-L_\phi-E_\phi}^{0}\sinh(L_\phi-\vert y_\phi(0,\eta)\vert) (4U_\phi^2U_{\phi,\xi}-4U_\phi Q_\phi y_{\phi,\xi}-2P_\phi U_{\phi,\xi})(0,\eta) \dee \eta\\
  &=\frac{1}{2\sinh(L_\phi)} \int_{0}^{L_\phi+E_\phi}(U_\phi(0,0)-U_\phi(0,\eta))\cosh( L_\phi-\vert y_\phi(0,\eta)\vert )\\
  &\qquad \qquad \qquad \qquad \qquad \qquad \qquad \qquad \times ((U_\phi^2+2mU_\phi-m^2)y_{\phi,\xi}+H_{\phi,\xi})(0,\eta)\dee \eta\\
  & \quad +\frac{1}{2\sinh(L_\phi)}\int_{0}^{L_\phi+E_\phi}\sinh(L_\phi-\vert y_\phi(0,\eta)\vert) (4U_\phi^2U_{\phi,\xi}-4U_\phi Q_\phi y_{\phi,\xi}-2P_\phi U_{\phi,\xi})(0,\eta) \dee \eta.
\end{align*}
Returning to the stumpon $\psi$ on the plateau, i.e., where $\xi\in (-(1+(s-m)^2)\ell, (1+(s-m)^2)\ell)$, we get 
\begin{align*}
  Q_{\psi,t}(0,\xi)&=\frac{1}{4\sinh(L_\psi)} \int_{-L_\psi-E_\psi}^{L_\psi+E_\psi}(U_\psi(0,\xi)-U_\psi(0,\eta))\cosh( L_\psi-\vert y_\psi(0,\xi)-y_\psi(0,\eta)\vert )\\
  &\qquad \qquad \qquad \qquad \qquad \qquad \qquad \qquad \times ((U_\psi^2+2mU_\psi-m^2)y_{\psi,\xi}+H_{\psi,\xi})(0,\eta) \dee \eta\\
  & \quad +\frac{1}{4\sinh(L_\psi)}\int_{-L_\psi-E_\psi}^{L_\psi+E_\psi}\sgn(\eta) \sinh(L_\psi-\vert y_\psi(0,\xi)-y_\psi(0,\eta)\vert) \\
  &\qquad \qquad \qquad \qquad \qquad \qquad \qquad \qquad \times (4U_\psi^2U_{\psi,\xi}-4U_\psi Q_\psi y_{\psi,\xi}-2P_\psi U_{\psi,\xi})(0,\eta) \dee \eta\\
  &=\frac{1}{4\sinh(L_\psi)} \int_{-L_\phi-E_\phi}^{L_\phi+E_\phi}(U_\phi(0,0)-U_\phi(0,\eta))\cosh( L_\phi-\sgn(\eta)( y_\psi(0,\xi)-y_\phi(0,\eta) )\\
  &\qquad \qquad \qquad \qquad \qquad \qquad \qquad \qquad \times ((U_\phi^2+2mU_\phi-m^2)y_{\phi,\xi}+H_{\phi,\xi})(0,\eta)\dee \eta\\
  & \quad +\frac{1}{4\sinh(L_\psi)}\int_{-L_\phi-E_\phi}^{L_\phi+E_\phi}\sgn(\eta) \sinh(L_\phi-\sgn(\eta)( y_\psi(0,\xi)-y_\phi(0,\eta)))\\
  &\qquad \qquad \qquad \qquad \qquad \qquad \qquad \qquad \times  (4U_\phi^2U_{\phi,\xi}-4U_\phi Q_\phi y_{\phi,\xi}-2P_\phi U_{\phi,\xi})(0,\eta) \dee \eta.
\end{align*}
Here we have used that the contribution from the plateau, i.e., $\abs{\eta} < (1+(s-m)^2)\ell$, vanishes because of \eqref{eq:plat:UH} and \eqref{eq:plat:PQ}.
Thereafter, we have applied the shift identities \eqref{eq:shift:y}, \eqref{eq:shift:UH}, and \eqref{eq:shift:PQ} to obtain an integral
over $[-L_\phi-E_\phi, L_\phi+E_\phi]$. Taking into account once more the symmetries \eqref{eq:sym} and using identities for hyperbolic functions we end up with 
\begin{align}\nonumber
  Q_{\psi,t}(0,\xi)
  &=\frac{\cosh(y_\psi(0,\xi))}{4\sinh(L_\psi)} \left[  \int_{-L_\phi-E_\phi}^{L_\phi+E_\phi}(U_\phi(0,0)-U_\phi(0,\eta))\cosh( L_\phi-\vert y_\phi(0,\eta)\vert )\right.\\ \nonumber
  &\qquad \qquad \qquad \qquad \qquad \qquad \qquad \qquad \times \left.((U_\phi^2+2mU_\phi-m^2)y_{\phi,\xi}+H_{\phi,\xi})(0,\eta)\dee \eta \right. \\ \nonumber
  & \quad+ \left. \int_{-L_\phi-E_\phi}^{L_\phi+E_\phi}\sgn(\eta) \sinh(L_\phi-\vert y_\phi(0,\eta)\vert) (4U_\phi^2U_{\phi,\xi}-4U_\phi Q_\phi y_{\phi,\xi}-2P_\phi U_{\phi,\xi})(0,\eta) \dee\eta \right] \\ \nonumber
  &= \frac{\sinh(L_\phi)\cosh(y_\psi(0,\xi))}{\sinh(L_\psi)}Q_{\phi,t}(0,0)\\   \label{eq:Qt:per}
  & = \frac{\sinh(L_\phi)\cosh(y_\psi(0,\xi))}{2\sinh(L_\phi+\ell)} (M-s)(s-m)(s-z),
\end{align}
where we recall that $L_\psi = L_\phi + \ell$.
The quantity $Q_{\psi,t}$ given by \eqref{eq:Qt:per} is either positive or negative, depending on whether we have an upward- or a downward-pointing stumpon initially,
i.e., whether the stumpon is ``constructed'' from a cuspon of type (e) or (e'), respectively, cf. Section~\ref{s:travel}.

Assume that we have an upward-pointing stumpon profile $\psi$ as initial data, i.e., the parameters satisfy $z<m<s<M$.
By the choice of the initial characteristics \eqref{eq:Lag:per:y}, we have that $y_\psi(0,\xi)=y_\psi(0,-\xi)$ for $\xi\in \R$ and $y_\psi(0,\cdot)$ is strictly increasing. Thus $Q_{\psi,t}(0,\xi)$ is strictly positive and continuous on the plateau,
and in particular it is strictly decreasing on $(-(1+(s-m)^2)\ell,0)$ and strictly increasing on $(0,(1+(s-m)^2)\ell)$. Furthermore, $Q_{\psi,t}(t,\xi)$ is continuous with respect to both time and space, so that we can write
\begin{align*}
  U_\psi(t,\xi)&=U_\psi(0,\xi)-tQ_\psi(0,\xi)-\frac{t^2}{2}Q_{\psi,t}(0,\xi)+\mathcal{O}(t^3)\\
  & = s-\frac{t^2}{2} Q_{\psi,t}(0,\xi)+\mathcal{O}(t^3),
\end{align*}
which implies that along characteristics, i.e., in Lagrangian coordinates, the plateau turns into a graph with a
hyperbolic cosine-like shape for small times, which in addition is moved downwards compared to the initial plateau height.
The numerical simulation in Figure \ref{fig:stumpon2} highlights this behavior.
We remark that in Eulerian coordinates, the shape will in general be influenced by the parametrization used for the
characteristics when going back from Lagrangian variables in accordance with Definition \ref{def:LtoE:per}.

Furthermore, a closer look at Figure~\ref{fig:stumpon2} reveals that the numerical solution is strictly increasing to the left
of the former plateau, while there emerges a small peak to the right of it. An explanation for this behavior is based on the prediction of wave breaking, which has been discussed in detail for the non-periodic setting in \cite[Theorem 1.1]{Grunert2015}.
We will briefly summarize the moral of the story: To the right of the plateau the initial data is strictly decreasing with an excessively negative slope. Thus one expects the solution to break along characteristics close and to the right of the plateau in the nearby future, and in particular the wave profile changes from decreasing to increasing, giving rise to the small spike. On the other hand, the initial data to the left of the plateau is strictly increasing with an especially positive slope.
This indicates that wave breaking occurred recently along characteristics close to the plateau and in particular the function cannot change from decreasing to increasing anytime soon. Hence there cannot turn up a spike to the left of the plateau.

\section{Conservative and non-conservative traveling waves on the real line} \label{s:cons:dec}
In this final section, we want to show that the results obtained in the periodic case also carry over to the cuspon and stumpon with decay. Since both the cuspon and the stumpon with decay have nonvanishing asymptotics, we have to combine results from \cite{HolRay2007cons, HolRay2007hyp,GruHolRay2012} to be able to apply the same methods as in the periodic case with only slight modifications.

On the real line, every conservative solution is given through a pair $(u,\mu)$ such that $(u(t),\mu(t))\in \Dcal$ for all $t\in \R$, where 
\begin{equation}
  \Dcal :=\{(u,\mu)\mid (u-m)\in H^1(\R), \enspace \mu\in \Mcal^+, \enspace \mu_\text{ac}=((u-m)^2+u_x^2)\dee x\}.
\end{equation}
Here $m$ denotes a constant, $\Mcal^+$ denotes the set of positive and finite Radon measures on $\R$, while $\mu_\text{ac}$ denotes the absolutely continuous part of a measure $\mu$. 

The system describing weak, conservative solutions reads
\begin{subequations}
  \begin{align}
    u_t+uu_x&=-P_x\\
    \mu_t+(u\mu)_x& =(u^3-mu^2+m^2u-2P(u-m))_x\\
    &=((u-m)^3+2m(u-m)^2+2m^2(u-m)+m^3-2P(u-m))_x,
  \end{align}
\end{subequations}
where 
\begin{equation}
  P(t,x)=\frac14 \int_\R \e^{-\vert x-z\vert} (u^2+2mu-m^2)(t,z) \dee z+\frac14 \int_\R \e^{-\vert x-z\vert} \dee\mu(t,z).
\end{equation}
As in the periodic case, the question arises whether or not cuspons and stumpons are conservative solutions in the sense of
\cite{GruHolRay2012}. 

The proof of Theorem~\ref{thm:noncon} does not rely on the periodicity, but only on the parameters $m$, $s$, and $M$ (recall that $z$ depends on the three other parameters). Thus keeping in mind that the only difference in parameters is that $z=m$ for the cuspon and stumpon with decay, one can follow the same lines to obtain the following result.

\begin{thm}
  Stumpons with decay, unlike cuspons with decay, are not conservative solutions of the Camassa--Holm equation, since they do not satisfy the conservation law \eqref{eq:claw} in the weak sense.
\end{thm}

Next we want to investigate how the conservative solution with stumpon initial data looks like. We will show the following result.
\begin{prp}
  Let $\psi$ be the decaying wave profile of a traveling wave stumpon with a plateau of positive width. Then, for the conservative solution of the Camassa--Holm equation with initial data $u_0(x)=\psi(x)$, the plateau will not persist.
\end{prp}

Again the proof relies heavily on the method of characteristics as introduced in \cite{HolRay2007hyp} and the ideas from the periodic case. Instead of presenting all the details we will give an outline, which uses representations that can be seen as the limit of the formulas from the periodic case, when $L_\phi\to \infty$ or equivalently $z\to m$.

\begin{dfn}
  For any $(u,\mu)\in \Dcal$, let $(y,U,H)$ on $\R$ be given by
  \begin{subequations}
    \begin{align}\label{eq:Lag:nper:y}
      y(\xi)&=\sup\{ x\in \R\mid x+\mu((-\infty,x))<\xi+\mu((-\infty,0])\},\\
      U(\xi)&= u(y(\xi)),\\
      H(\xi)&=\xi-y(\xi).
    \end{align}
  \end{subequations}
\end{dfn}

The evolution equations yielding conservative solutions are then given by
\begin{subequations}
  \begin{align}
    y_t&=U,\\
    U_t&=-Q,\\
    H_t& = U^3-mU^2+m^2U-2P(U-m)\\
    &= (U-m)^3+2m(U-m)^2+2m^2(U-m)+m^3-2P(U-m),
  \end{align}
\end{subequations}
where 
\begin{subequations}
  \begin{align}
    P(t,\xi)&=\frac14 \int_\R \e^{-\vert y(t,\xi)-y(t,\eta)\vert} ((U^2+2mU-m^2)y_\xi+H_\xi)(t,\eta) \dee\eta,\\
    Q(t,\xi)&= -\frac14 \int_\R \sgn(\xi-\eta)\e^{-\vert y(t,\xi)-y(t,\eta)\vert} ((U^2+2mU-m^2)y_\xi+H_\xi)(t,\eta) \dee\eta.
  \end{align}
\end{subequations}

Going back to Eulerian coordinates, one can apply the following mapping, which is taken from \cite{GruHolRay2012}.
\begin{dfn}
  Given $(y,U,H)$, define $(u,\mu)\in \Dcal$ as follows
  \begin{align*}
    u(x)&=U(\xi) \text{ for any } \xi \text{ such that } x=y(\xi),\\
    \mu&=y_{\#}(H_\xi \dee \xi).
  \end{align*}
\end{dfn}

Following the same lines as in the last section, one obtains that the equalities \eqref{eq:shift:y}--\eqref{eq:plat:PQ} also hold for the stumpon and the cuspon with decay, if one replaces the interval $-L_\phi-E_\phi\leq \xi\leq 0$ and $0\leq \xi\leq L_\phi+E_\phi$ with $\xi\leq 0$ and $0\leq \xi$, respectively. 

Following once more the proof of existence and continuity of $Q_t$ for the cuspon with decay in \cite{Grunert2016}, one can show that $Q_t$ exists and is given by 
\begin{align*}
  Q_t(t,\xi)&=\frac14 \int_\R (U(t,\xi)-U(t,\eta))\e^{-\vert y(t,\xi)-y(t,\eta)\vert}((U^2+2mU-m^2)y_\xi+H_\xi)(t,\eta)\dee\eta\\
  & \quad - \frac14 \int_\R \sgn(\xi-\eta) \e^{-\vert y(t,\xi)-y(t,\eta)\vert} (4U^2U_\xi-4UQy_\xi-2PU_\xi)(t,\eta) \dee\eta.
\end{align*}
In particular, one has, cf.\ \cite{Grunert2016}, that the cuspon with decay satisfies 
\begin{equation*}
  Q_{\phi,t}(0,0)=\frac12 (M-s)(s-m)^2\not =0.
\end{equation*}
Again we would like to find a relation between $Q_{\phi,t}(0,0)$ and $Q_{\psi,t}(0,\xi)$ for $\xi \in (-(1+(s-m)^2)\ell, (1+(s-m)^2)\ell)$. In this case we note that the symmetry properties \eqref{eq:sym} are valid for $k\in \R$ and we can therefore write
\begin{align*}
  Q_{\phi,t}(0,0)&= \frac12 \int_{-\infty}^0 (U_\phi(0,0)-U_\phi(0,\eta))\e^{ y_\phi(0,\eta)}((U_\phi^2+2mU_\phi-m^2)y_{\phi,\xi}+H_{\phi,\xi})(0,\eta)\dee\eta\\
  & \quad - \frac12 \int_{-\infty}^0  \e^{y_\phi(0,\eta)} (4U_\phi^2U_{\phi,\xi}-4U_\phi Q_\phi y_{\phi,\xi}-2P_\phi U_{\phi,\xi})(0,\eta) \dee\eta\\
  & =  \frac12 \int_0^\infty (U_\phi(0,0)-U_\phi(0,\eta)) \e^{-y_\phi(0,\eta)}((U_\phi^2+2mU_\phi-m^2)y_{\phi,\xi}+H_{\phi,\xi})(0,\eta)\dee\eta\\
  & \quad - \frac12 \int_0^\infty  \e^{-y_\phi(0,\eta)} (4U_\phi^2U_{\phi,\xi}-4U_\phi Q_\phi y_{\phi,\xi}-2P_\phi U_{\phi,\xi})(0,\eta) \dee\eta.
\end{align*}
Carrying out the same calculations as in the periodic case with only slight modifications, we end up with 
\begin{align}\nonumber
  Q_{\psi,t}(0,\xi)& = \e^{-\ell}\cosh(y_\psi(0,\xi))Q_{\phi,t}(0,0)\\ \label{eq:Qt:line}
  & = \frac{\e^{-\ell}\cosh(y_\psi(0,\xi))}{2}(M-s)(s-m)^2,
\end{align}
which is either positive or negative, depending on whether we have an upward- or a downward-pointing stumpon initially,
i.e., whether the stumpon is ``constructed'' from a cuspon of type (f) or (f'), respectively.

We underline that it is not only the above calculations which are very similar for the stumpon with decay and the periodic stumpon.
As mentioned earlier, the decaying stumpon can be seen as the limit of the periodic case, by keeping $\ell$ fixed in $L_\psi = L_\phi+\ell$ and either letting $L_\phi \to \infty$ or $z\to m$.
To be precise, we consider the half-period $L_\phi$ for the cuspon given by \eqref{eq:L:cusp}.
If $L_\phi$ is to be unbounded, the final integral must diverge, but for all cuspons one has $\abs{M-\phi}^{-1} \ge \abs{M-s}^{-1} > 0$. Therefore, the only way this can happen is to have $z \to m$. Furthermore, we can fix a point $x$ away from the plateau so that the identity \eqref{eq:dphi2Cper} holds.
Then, as sending $L_\phi \to \infty$ sends $z \to m$, we of course obtain \eqref{eq:dphi2Cdec},
so the slope of the wave profile must also agree in the limit.

There are also other quantities for the two types of stumpons which agree in the limit.
Recall that we can write $Q_{\psi,t}(0,\xi)$ on the plateau in the periodic case, cf. \eqref{eq:Qt:per}, as 
\begin{align*}
  Q_{\psi,t}(0,\xi)&=\frac12\frac{\e^{L_\phi}-\e^{-L_\phi}}{\e^{L_\phi+\ell}-\e^{-L_\phi-\ell}}\cosh(y_\psi(0,\xi))(M-s)(s-m)(s-z)\\
  & = \frac12 \e^{-\ell} \frac{1-\e^{-2L_\phi}}{1-\e^{-2L_\phi-2\ell}}\cosh(y_\psi(0,\xi)) (M-s)(s-m)(s-z).
\end{align*}
This expression tends to \eqref{eq:Qt:line} as $L_\phi \to \infty$, $z \to m$, and  the periodic initial characteristics $y_\psi(0,\xi)$ given through \eqref{eq:Lag:per:y} tend to the non-periodic ones defined by \eqref{eq:Lag:nper:y}. Thus one can expect the plateau to
evolve rather similarly for the two cases.

To summarize, from a qualitative point of view the behavior of the weak conservative solution for periodic and decaying stumpon initial data should be rather similar. 

\subsection*{Acknowledgments:}
The authors thank Jonatan Lenells,
whose seemingly innocuous question led to the discovery of these results.


\begin{thebibliography}{10}
	
	\bibitem{BreCheZha2015}
	Alberto Bressan, Geng Chen, and Qingtian Zhang.
	\newblock Uniqueness of conservative solutions to the {C}amassa-{H}olm equation
	via characteristics.
	\newblock {\em Discrete Contin. Dyn. Syst.}, 35(1):25--42, 2015.
	\newblock \href {https://doi.org/10.3934/dcds.2015.35.25}
	{\path{doi:10.3934/dcds.2015.35.25}}.
	
	\bibitem{BreCon2007cons}
	Alberto Bressan and Adrian Constantin.
	\newblock Global conservative solutions of the {C}amassa-{H}olm equation.
	\newblock {\em Arch. Ration. Mech. Anal.}, 183(2):215--239, 2007.
	\newblock \href {https://doi.org/10.1007/s00205-006-0010-z}
	{\path{doi:10.1007/s00205-006-0010-z}}.
	
	\bibitem{BreCon2007diss}
	Alberto Bressan and Adrian Constantin.
	\newblock Global dissipative solutions of the {C}amassa-{H}olm equation.
	\newblock {\em Anal. Appl. (Singap.)}, 5(1):1--27, 2007.
	\newblock \href {https://doi.org/10.1142/S0219530507000857}
	{\path{doi:10.1142/S0219530507000857}}.
	
	\bibitem{Camassa2003}
	Roberto Camassa.
	\newblock Characteristics and the initial value problem of a completely
	integrable shallow water equation.
	\newblock {\em Discrete Contin. Dyn. Syst. Ser. B}, 3(1):115--139, 2003.
	\newblock \href {https://doi.org/10.3934/dcdsb.2003.3.115}
	{\path{doi:10.3934/dcdsb.2003.3.115}}.
	
	\bibitem{CamHol1993}
	Roberto Camassa and Darryl~D. Holm.
	\newblock An integrable shallow water equation with peaked solitons.
	\newblock {\em Phys. Rev. Lett.}, 71(11):1661--1664, 1993.
	\newblock \href {https://doi.org/10.1103/PhysRevLett.71.1661}
	{\path{doi:10.1103/PhysRevLett.71.1661}}.
	
	\bibitem{ConMcK1999}
	A.~Constantin and H.~P. McKean.
	\newblock A shallow water equation on the circle.
	\newblock {\em Comm. Pure Appl. Math.}, 52(8):949--982, 1999.
	\newblock \href
	{https://doi.org/10.1002/(SICI)1097-0312(199908)52:8<949::AID-CPA3>3.0.CO;2-D}
	{\path{doi:10.1002/(SICI)1097-0312(199908)52:8<949::AID-CPA3>3.0.CO;2-D}}.
	
	\bibitem{ConEsc1998}
	Adrian Constantin and Joachim Escher.
	\newblock Global existence and blow-up for a shallow water equation.
	\newblock {\em Ann. Scuola Norm. Sup. Pisa Cl. Sci. (4)}, 26(2):303--328, 1998.
	\newblock URL: \url{http://www.numdam.org/item?id=ASNSP_1998_4_26_2_303_0}.
	
	\bibitem{ConKol2003}
	Adrian Constantin and Boris Kolev.
	\newblock Geodesic flow on the diffeomorphism group of the circle.
	\newblock {\em Comment. Math. Helv.}, 78(4):787--804, 2003.
	\newblock \href {https://doi.org/10.1007/s00014-003-0785-6}
	{\path{doi:10.1007/s00014-003-0785-6}}.
	
	\bibitem{ConLan2009}
	Adrian Constantin and David Lannes.
	\newblock The hydrodynamical relevance of the {C}amassa-{H}olm and
	{D}egasperis-{P}rocesi equations.
	\newblock {\em Arch. Ration. Mech. Anal.}, 192(1):165--186, 2009.
	\newblock \href {https://doi.org/10.1007/s00205-008-0128-2}
	{\path{doi:10.1007/s00205-008-0128-2}}.
	
	\bibitem{ConStr2000}
	Adrian Constantin and Walter~A. Strauss.
	\newblock Stability of peakons.
	\newblock {\em Comm. Pure Appl. Math.}, 53(5):603--610, 2000.
	\newblock \href
	{https://doi.org/10.1002/(SICI)1097-0312(200005)53:5<603::AID-CPA3>3.3.CO;2-C}
	{\path{doi:10.1002/(SICI)1097-0312(200005)53:5<603::AID-CPA3>3.3.CO;2-C}}.
	
	\bibitem{EckKos2014}
	Jonathan Eckhardt and Aleksey Kostenko.
	\newblock An isospectral problem for global conservative multi-peakon solutions
	of the {C}amassa-{H}olm equation.
	\newblock {\em Comm. Math. Phys.}, 329(3):893--918, 2014.
	\newblock \href {https://doi.org/10.1007/s00220-014-1905-4}
	{\path{doi:10.1007/s00220-014-1905-4}}.
	
	\bibitem{EckKos2018}
	Jonathan Eckhardt and Aleksey Kostenko.
	\newblock The inverse spectral problem for periodic conservative multi-peakon
	solutions of the {C}amassa-{H}olm equation.
	\newblock {\em Int. Math. Res. Not. IMRN}, 2020(16):5126--5151, 2020.
	\newblock \href {https://doi.org/10.1093/imrn/rny176}
	{\path{doi:10.1093/imrn/rny176}}.
	
	\bibitem{FucFok1981}
	B.~Fuchssteiner and A.~S. Fokas.
	\newblock Symplectic structures, their {B}\"{a}cklund transformations and
	hereditary symmetries.
	\newblock {\em Phys. D}, 4(1):47--66, 1981.
	\newblock \href {https://doi.org/10.1016/0167-2789(81)90004-X}
	{\path{doi:10.1016/0167-2789(81)90004-X}}.
	
	\bibitem{GalGru2021}
	Sondre~Tesdal Galtung and Katrin Grunert.
	\newblock A numerical study of variational discretizations of the
	{C}amassa-{H}olm equation.
	\newblock {\em BIT}, 61(4):1271--1309, 2021.
	\newblock \href {https://doi.org/10.1007/s10543-021-00856-1}
	{\path{doi:10.1007/s10543-021-00856-1}}.
	
	\bibitem{GalRay2021}
	Sondre~Tesdal Galtung and Xavier Raynaud.
	\newblock A semi-discrete scheme derived from variational principles for global
	conservative solutions of a {C}amassa-{H}olm system.
	\newblock {\em Nonlinearity}, 34(4):2220--2274, 2021.
	\newblock \href {https://doi.org/10.1088/1361-6544/abc101}
	{\path{doi:10.1088/1361-6544/abc101}}.
	
	\bibitem{Grunert2015}
	Katrin Grunert.
	\newblock Blow-up for the two-component {C}amassa-{H}olm system.
	\newblock {\em Discrete Contin. Dyn. Syst.}, 35(5):2041--2051, 2015.
	\newblock \href {https://doi.org/10.3934/dcds.2015.35.2041}
	{\path{doi:10.3934/dcds.2015.35.2041}}.
	
	\bibitem{Grunert2016}
	Katrin Grunert.
	\newblock Solutions of the {C}amassa-{H}olm equation with accumulating breaking
	times.
	\newblock {\em Dyn. Partial Differ. Equ.}, 13(2):91--105, 2016.
	\newblock \href {https://doi.org/10.4310/DPDE.2016.v13.n2.a1}
	{\path{doi:10.4310/DPDE.2016.v13.n2.a1}}.
	
	\bibitem{GruHolRay2012}
	Katrin Grunert, Helge Holden, and Xavier Raynaud.
	\newblock Global conservative solutions to the {C}amassa-{H}olm equation for
	initial data with nonvanishing asymptotics.
	\newblock {\em Discrete Contin. Dyn. Syst.}, 32(12):4209--4227, 2012.
	\newblock \href {https://doi.org/10.3934/dcds.2012.32.4209}
	{\path{doi:10.3934/dcds.2012.32.4209}}.
	
	\bibitem{GruHolRay2014}
	Katrin Grunert, Helge Holden, and Xavier Raynaud.
	\newblock Global dissipative solutions of the two-component {C}amassa-{H}olm
	system for initial data with nonvanishing asymptotics.
	\newblock {\em Nonlinear Anal. Real World Appl.}, 17:203--244, 2014.
	\newblock \href {https://doi.org/10.1016/j.nonrwa.2013.12.001}
	{\path{doi:10.1016/j.nonrwa.2013.12.001}}.
	
	\bibitem{HolRay2007MP}
	Helge Holden and Xavier Raynaud.
	\newblock Global conservative multipeakon solutions of the {C}amassa-{H}olm
	equation.
	\newblock {\em J. Hyperbolic Differ. Equ.}, 4(1):39--64, 2007.
	\newblock \href {https://doi.org/10.1142/S0219891607001045}
	{\path{doi:10.1142/S0219891607001045}}.
	
	\bibitem{HolRay2007cons}
	Helge Holden and Xavier Raynaud.
	\newblock Global conservative solutions of the {C}amassa-{H}olm equation---a
	{L}agrangian point of view.
	\newblock {\em Comm. Partial Differential Equations}, 32(10-12):1511--1549,
	2007.
	\newblock \href {https://doi.org/10.1080/03605300601088674}
	{\path{doi:10.1080/03605300601088674}}.
	
	\bibitem{HolRay2007hyp}
	Helge Holden and Xavier Raynaud.
	\newblock Global conservative solutions of the generalized hyperelastic-rod
	wave equation.
	\newblock {\em J. Differential Equations}, 233(2):448--484, 2007.
	\newblock \href {https://doi.org/10.1016/j.jde.2006.09.007}
	{\path{doi:10.1016/j.jde.2006.09.007}}.
	
	\bibitem{HolRay2008}
	Helge Holden and Xavier Raynaud.
	\newblock Periodic conservative solutions of the {C}amassa-{H}olm equation.
	\newblock {\em Ann. Inst. Fourier (Grenoble)}, 58(3):945--988, 2008.
	\newblock URL: \url{http://aif.cedram.org/item?id=AIF_2008__58_3_945_0}.
	
	\bibitem{HolRay2009diss}
	Helge Holden and Xavier Raynaud.
	\newblock Dissipative solutions for the {C}amassa-{H}olm equation.
	\newblock {\em Discrete Contin. Dyn. Syst.}, 24(4):1047--1112, 2009.
	\newblock \href {https://doi.org/10.3934/dcds.2009.24.1047}
	{\path{doi:10.3934/dcds.2009.24.1047}}.
	
	\bibitem{Johnson2002}
	R.~S. Johnson.
	\newblock Camassa-{H}olm, {K}orteweg-de {V}ries and related models for water
	waves.
	\newblock {\em J. Fluid Mech.}, 455:63--82, 2002.
	\newblock \href {https://doi.org/10.1017/S0022112001007224}
	{\path{doi:10.1017/S0022112001007224}}.
	
	\bibitem{KalLen2005}
	Henrik Kalisch and Jonatan Lenells.
	\newblock Numerical study of traveling-wave solutions for the {C}amassa--{H}olm
	equation.
	\newblock {\em Chaos Solitons Fractals}, 25(2):287--298, 2005.
	\newblock URL: \url{http://dx.doi.org/10.1016/j.chaos.2004.11.024}, \href
	{https://doi.org/10.1016/j.chaos.2004.11.024}
	{\path{doi:10.1016/j.chaos.2004.11.024}}.
	
	\bibitem{Kouranbaeva1999}
	Shinar Kouranbaeva.
	\newblock The {C}amassa-{H}olm equation as a geodesic flow on the
	diffeomorphism group.
	\newblock {\em J. Math. Phys.}, 40(2):857--868, 1999.
	\newblock \href {https://doi.org/10.1063/1.532690}
	{\path{doi:10.1063/1.532690}}.
	
	\bibitem{Lenells2004}
	Jonatan Lenells.
	\newblock A variational approach to the stability of periodic peakons.
	\newblock {\em J. Nonlinear Math. Phys.}, 11(2):151--163, 2004.
	\newblock \href {https://doi.org/10.2991/jnmp.2004.11.2.2}
	{\path{doi:10.2991/jnmp.2004.11.2.2}}.
	
	\bibitem{Lenells2005cons}
	Jonatan Lenells.
	\newblock Conservation laws of the {C}amassa-{H}olm equation.
	\newblock {\em J. Phys. A}, 38(4):869--880, 2005.
	\newblock \href {https://doi.org/10.1088/0305-4470/38/4/007}
	{\path{doi:10.1088/0305-4470/38/4/007}}.
	
	\bibitem{Lenells2005}
	Jonatan Lenells.
	\newblock Traveling wave solutions of the {C}amassa--{H}olm equation.
	\newblock {\em J. Differential Equations}, 217(2):393--430, 2005.
	\newblock URL: \url{http://dx.doi.org/10.1016/j.jde.2004.09.007}, \href
	{https://doi.org/10.1016/j.jde.2004.09.007}
	{\path{doi:10.1016/j.jde.2004.09.007}}.
	
\end{thebibliography}

\end{document}